\documentclass[]{imsart}

\RequirePackage[OT1]{fontenc}
\RequirePackage{amsthm,amsmath,amssymb}
\usepackage{epsfig}
\usepackage{graphicx}
\usepackage{epstopdf,color}
\usepackage[usenames,dvipsnames,svgnames,table]{xcolor}
\usepackage{multirow}
\usepackage{listings}
\RequirePackage{natbib}
\RequirePackage[colorlinks,citecolor=blue,urlcolor=blue]{hyperref}
\allowdisplaybreaks[4]
\theoremstyle{plain}
\numberwithin{equation}{section}
\numberwithin{figure}{section}
\usepackage{xr}
\startlocaldefs
\numberwithin{equation}{section}
\theoremstyle{plain}
\newtheorem{lemma}{Lemma}[section]
\newtheorem{theorem}{Theorem}[section]

\endlocaldefs

\newcommand{\B}{{\mathbf B}}

\newcommand{\y}{{\mathbf y}}

\newcommand{\E}{{\mathbb{E}}}

\newcommand{\tr}{{\mathop{\text{\rm tr}}}}

\definecolor{darkblue}{rgb}{0,0.08,0.45}

\renewcommand{\(}{\left(}
\renewcommand{\)}{\right)}

\newcommand{\var}{{\rm Var}}

\newcommand{\mb}{\mathbf}
\newcommand{\du}{\circ}
\newcommand{\di}{{\rm Diag}}

\begin{document}

\begin{frontmatter}
\title{Spectral statistics of high dimensional sample covariance matrix with unbounded population spectral norm}
	\runtitle{Spectral statistics of sample covariance matrix}
\author{Yanqing Yin}
\runauthor{Yanqing Yin}
\begin{abstract}
In this paper, we establish some new central limit theorems for certain spectral statistics of a high-dimensional sample covariance matrix under a divergent spectral norm population model. This model covers the divergent spiked population model as a special case. Meanwhile, the number of the spiked eigenvalues can either be fixed or grow to infinity. It is seen from our theorems that the divergence of population spectral norm affects the fluctuations of the linear spectral statistics in a fickle way, depending on the divergence rate.
\end{abstract}

\end{frontmatter}

\section{Introduction}
Linear spectral statistics of large sample covariance matrices play important roles in large-scale statistical inference.
Let $\y _1,\ldots,\y _n$ be $n$ independent observations from the population $\y \in\mathbb R^p$ with zero mean and covariance matrix $\Sigma$.
The sample covariance matrix of the observations is
$$
\B_n=\frac{1}{n}\sum_{j=1}^n\y _j\y _j'.
$$
Denote the eigenvalues of $\B_n$ as $\lambda_1\geq \lambda_2\geq\cdots\geq\lambda_p$.
Then, for a test function $f$ defined on $\mathbb R$, its associated {\em linear spectral statistic} (LSS) \citep{BSbook} of $\B_n$ is of the form
\begin{equation}\label{lss}
\frac{1}{p}\sum_{i=1}^pf(\lambda_i)=\int f(x)dF^{\B_n}(x),
\end{equation}
where $F^{\B_n}=(1/p)\sum_{i=1}^p\delta_{\lambda_i}$ is the {\em spectral distribution} (SD) of $\B_n$ and $\delta_{a}$ is the Dirac measure at the point $a$.

In view of the wide application in statistics, the fluctuation of the LSS has been investigated by many authors under high dimensional regime, where the dimension $p$ of the observations grows at the same rate as the sample size $n$ such that $p/n\to c\in (0, \infty)$. \cite{jonsson1982some} firstly considers the fluctuation of LSSs associated with the polynomial test functions under the null case where the entries of $\mb y$ are i.i.d.. For better applications in statistics, \cite{BS04} obtain the {\em central limit theorem} (CLT) for LSS associated with test functions that are analytic beyond the null case. However,
% the population $\y $ is modeled by a structure of independent components, that is,
%\begin{align}\label{ICM}
%	\y =\mb\Sigma^{\frac12}\mb x,
%\end{align}
%where $\E \y=0$ and $\mb\Sigma$ are the population mean vector and covariance matrix, respectively, and $\mb x=(x_1,\ldots,x_p)'$ is an array of  i.i.d.\ normalized random variables with finite fourth moment. What is more, 
they need a Gaussian like fourth moment assumption.  This assumption is further relaxed in \cite{PZ08,zheng2015substitution}. Other extensions on this topic can be found in \cite{gao2017high,huli19}. We also refer the readers to \cite{ledoit2002some,Schott05,srivastava2005some,yang15}, etc., for applications of LSSs in large-scale statistical inference. 

A critical assumption made in the above references is that the spectral norm of the covariance matrix $\mb\Sigma$ need to be bounded uniformly in $p$. Thus, those results are excluded from many applications in various fields such as finance and economics, where a group of leading eigenvalues of $\mb\Sigma$ may diverge to infinity as the increase of the dimension $p$, see \cite{baltagi2017asymptotic}. 
In the light of this fact, we investigate in this paper the joint asymptotic distribution of LSSs of $\B_n$ with polynomial test functions when the spectral norm of $\mb\Sigma$ may diverge.
The results show that the joint distributions of the spectral statistics are still asymptotically Gaussian under suitable moment conditions but  their limiting mean vector and covariance matrix are quite different from the existing results under the assumption of bounded spectral norm on $\mb\Sigma$. A new feature of the proposed CLT is that the main terms of the limiting covariance matrix may vary depending on the divergence order of the spectral norm.
This characterizes how the spectral norm of $\mb\Sigma$ contributes to the fluctuations of LSSs.

The remaining parts of the paper are organized as follows. Section \ref{sec:clt} establishes the new CLT for LSSs of $\B_n$ with polynomial test functions under the divergent spiked population model. Some simulations are presented in Section \ref{sec:sim}.
Technical proofs of the main theorems are postponed to Section~\ref{sec:proofs}-\ref{sec:proof3}. 

In the rest of this paper, we use $\mb 1_n$ denote the $n$ dimension vector whose entries all equal 1. For integer $k$, let $n_{(k)}=n!/(n-k)!$.
 Use $\di\(\mb B\)=\(b_{1,1},b_{2,2},\cdots,b_{n,n}\)'$ to stand for the vector formed by the diagonal entries of $\mb B$, use $\mb D_{\mb B}$ to denote the diagonal matrix of $\mb B$ (replacing all off-diagonal entries with zero).  We also use $\mb A \du \mb B=(a_{ij}b_{ij})$ to denote the Hadamard product of two matrices $\mb A=(a_{ij})$ and $\mb B=(b_{ij})$ and use $\mb A ^{\du k}$ to denote the Hadamard product of $k$ matrices $\mb A$. For two sequences $a_n$ and $b_n$, we use $a_n\simeq b_n$ to stand for $a_n=O(b_n)$ and $b_n=O(a_n)$ as $n\to\infty$.

\section{Model assumptions and main theorems}
\label{sec:clt}
This section is to present our model assumptions and main theorems. We first introduce the following {\em divergent spectral norm population model} on the population covariance matrix $\mb\Sigma$. Specifically,
let  $(\tau_j)$ be the eigenvalues of $\mb\Sigma$ arranging in descending order, which are grouped into two classes $G_1$ and $G_2$, i.e.
\begin{align*}
	{\bf spec}(\mb\Sigma)= \underbrace{\{\tau_1,\ldots\tau_{k_p}\}}_{G_1}\cup\underbrace{\{\tau_{k_p+1},\ldots,\tau_p\}}_{G_2},
\end{align*}
satisfying
\begin{equation*}
	\inf_p  \min\limits_{1\le j\le k_p} \tau_j\to\infty,\ \inf_p  \min\limits_{k_p\le j\le p} \tau_j>0 \quad
	\text{and}\quad
	\sup_p \max\limits_{k_p\le j\le p} \tau_j<\infty.
\end{equation*}
The above model includes the so-called {\em spiked population model}, which is originally introduced in \cite{Johnstone01}, as a special case.
%We refer to this model as {\em divergent spiked population model}. 
Note that the number of divergent eigenvalues is allowed to increase to infinity under such model.
Now we consider the joint CLT for LSSs of the sample covariance matrix $\B_n$ associated with test functions $f_\ell(x)=x^\ell$, $\ell=1,\ldots,m$, say
$$
T_\ell\triangleq \sum_{i=1}^pf_\ell(\lambda_i)=\tr(\B_n^\ell),\quad \ell=1,\ldots,m.
$$

We introduce the following model assumptions. %CLT of $(T_\ell)_{1\leq \ell\leq m}$.

\begin{description}
	\item[Assumption A:]
	As  $n\to \infty$,  $p=p_n\to\infty$ such that $p/n\to c\in (0,\infty)$. 
	\item[Assumption B:] The population $\y $ follows the independent components model, that is, $\y =\mb\mu+\mb\Sigma^{\frac12}\mb x$ with $\mb x=(x_1,\ldots,x_p)'$ being an array of i.i.d.\ random variables satisfying
	\begin{equation}\label{m-con}
		\E(x_{1})=0,\quad \E(x_{1}^2)=1,\quad \E(x_{1}^4)=\nu_4+3\quad \text{and}\quad \E(|x_{1}|^\gamma)<\infty,
	\end{equation}
	for some constant $\gamma>4$.
	\item[Assumption C:] The eigenvalues in $G_1$ satisfy $\max_{1\leq j\leq k_p}\tau_j=O(\log p)$.
		\item[Assumption C':] The eigenvalues in $G_1$ satisfy $\max_{1\leq j\leq k_p}\tau_j\simeq p^\alpha $ and the size of $G_1$ is $k_p/n\to\beta$ for some $\alpha\geq 0$ and $0\leq \beta\leq 1$.
\end{description}
Assumption A and B are commonly assumed in random matrix theory, especially in the studies of CLT for LSS.
 We must note that $\gamma=4$ is in general sufficient to ensure the CLT for LSS when the spectral norm of $\mb\Sigma$ is bounded, as can be seen from the previously mentioned references. However, higher-order moments are needed when $\|\mb\Sigma\|$ diverges. 
 
 Our first theorem is then described as follows.
\begin{theorem}\label{kjoint}
	Suppose that the Assumptions A-C hold. Then, for any fixed $m$, the random vector
	$$
	\mb \Psi_m^{-\frac12}\left(T_1-\E T_1,T_2-\E T_2,\cdots, T_m-\E T_m\right)'\xrightarrow {D} N_m(0, \mb I_m),
	$$	
where $\mb\Psi_m$ denotes the covariance matrix of $(T_1,\ldots,T_m)'$.
\end{theorem}

Theorem \ref{kjoint} illustrates that when the spectral norm of $\mb\Sigma$ diverges to infinity at the rate of $O(\log(p))$, the joint distribution of LSSs with polynomial test functions still obeys the normal law asymptotically. In this point, it generalizes the classic results built in \cite{jonsson1982some,BS04,PZ08,zheng2015substitution}.
Note that the implementation of this theorem needs the knowledge of the centralization terms $(\E(T_j))$ and the covariance matrix $\mb\Psi_m$, which have no explicit and unified expressions at present. However, this only involves routine calculations for given $m$.

When the variables $(z_j)$ possess finite higher-order moments, we have the following theorem.

\begin{theorem}\label{2joint}
	Suppose that  the Assumptions A-C' hold. In addition, Assumptions B holds with $\gamma=6$ when { $4\alpha+\beta\leq 1$} and with $\gamma=16$ when { $4\alpha+\beta>1$}, then the random vector
	$$
	\mb\Psi_2^{-\frac12}\(T_1-\E T_1,T_2-\E T_2\)'\xrightarrow {D}N_2(0,\mb I_2),
	$$	
where the expectations  are
	$$\E T_1=\tr (\mb \Sigma)\quad\text{and}\quad \E T_2=n^{-1}\left[\nu_4\tr\(\mb \Sigma\du\mb \Sigma\)+{\tr}^2 (\mb \Sigma)+(n+1)\tr\(\mb \Sigma^2\)\right],
	$$
and the covariance matrix is $\mb\Psi_2=(\psi_{ij})_{2\times 2}$ with its entries
	\begin{align*}
	\psi_{11}&=n^{-1}\left[\nu_4\tr\(\mb \Sigma^{\du2}\)+2\tr\(\mb \Sigma^2\)\right],\\
	\psi_{12}&=\psi_{21}=n^{-2}\left[4\tr(\mb \Sigma^2)\tr (\mb \Sigma)+2\nu_4\tr\(\mb \Sigma^{\du2}\)\tr \mb \Sigma+2\nu_4n\tr\(\mb \Sigma\du\mb \Sigma^2\)+4n\tr(\mb \Sigma^3)\right],\\
		\psi_{22}&=n^{-3}\Big[8\tr(\mb \Sigma^2){\tr}^2(\mb \Sigma)+4\nu_4{\tr}^2(\mb \Sigma)\tr\(\mb \Sigma^{\du2}\)+16n\tr(\mb \Sigma)\tr(\mb \Sigma^3)\\\notag
		&\quad+4n{\tr}^2(\mb \Sigma^2)+8\nu_4n\tr\(\mb \Sigma\du\mb \Sigma^2\)\tr\mb \Sigma
		+4\nu_4n^2\tr\(\mb \Sigma^2\du\mb \Sigma^2\)+8n^2\tr(\mb \Sigma^4)\Big].
	\end{align*}
\end{theorem}

In many statistical applications such as tests for covariance structure, statistics $T_1$ and $T_2$ play important roles. Theorem \ref{2joint} demonstrates the asymptotic Gaussianity of the vector $(T_1, T_2)'$ after normalization, where the spectral norm of $\mb \Sigma$ diverge at the rate of $n^{\alpha}$. In addition, exact formulae of the expectations and the covariance matrix are given. These formulae exhibit the way how the population spectrum contributes to the limiting distribution of the spectral statistics.
For instance, we analyze the statistic $T_2$ under a simplified spiked population model \citep{Johnstone01}, where the covariance matrix $\mb\Sigma$ is diagonal and has only one spiked eigenvalue, i.e.
\begin{align*}
\mb\Sigma= {\rm diag}(\tau_1, 1,\ldots,1).
\end{align*}
We then have $\E T_2=[(\tau_1+p-1)^2+(n+1+\nu_4)(\tau_1+p-1)]/n$ and
\begin{description}
	\item[Case 1.] when $\tau_1=O(1)$,
$$T_2-\E T_2\xrightarrow {D}N\big(0,4 c (2 + 5 c + 2 c^2) + 4 c (1 + 2 c + c^2) \nu_4\big),$$
	\item[Case 2.] when $\tau_1=O(p^{\frac14})$,
$$T_2-\E T_2\xrightarrow {D}N\big(0,4 c (2 + 5 c + 2 c^2) + 4 c (1 + 2 c + c^2) \nu_4+\delta^4c(8+4\nu_4)\big),,$$
	\item[Case 3.] when $\frac{\tau_1^4}{p}\to\infty$,
$$\frac{\sqrt{n}}{\tau_1^2}(T_2-\E T_2)\xrightarrow {D}N\big(0,8+4\nu_4\big),$$
\end{description}

where the constant $\delta=\lim\tau_1/p^{1/4}$ in the second asymptotic variance.
These results reveal that, when the single spike is weak (Case 1), it only results in a mean shift in the asymptotic distribution of an LSS, which is consistent with the main theorem in \cite{BS04}. As the spike becomes stronger, it will gradually lead to shifts of both mean and variance of an LSS (Case 2), and finally dominates its asymptotic distribution (Case 3). 

If the number of divergent eigenvalues also tends to infinity, that is, $\beta>0$, then the leading terms in the expressions in Theorem \ref{2joint} are not apparent and need to identify carefully. This phenomenon exhibits 
 that the researches is non-trivial and full of difficulties when $\|\mb \Sigma\|$ diverges.
 
In practice, the mean vector of the population is often unknown, and we shall use the revised sample covariance matrix $$
\B_n^0=\frac{1}{n-1}\sum_{j=1}^n(\y _j-\bar \y _j)(\y _j'-\bar \y _j),
$$
where $\bar \y _j=\frac1n\sum_{j=1}^n\y _j$ is the sample mean. Then we have the following theorem concerning the joint distribution of $
T_1^0=\tr(\B_n^0)$ and $
T_2^0=\tr(\(\B_n^0\)^2)$.
\begin{theorem}\label{2jointm}
Suppose that the conditions in Theorem \ref{2joint} are satisfied, then the random vector
	$$
	\mb\Psi_2^{-\frac12}\(T_1^0-\E T_1^0,T_2^0-\E T_2^0\)'\xrightarrow {D}N_2(0,\mb I_2),
	$$	
where the expectations  are
	$$\E T_1^0=\tr (\mb \Sigma)+\frac1n\tr (\mb \Sigma)$$ and $$\E T_2^0=n^{-1}\left[\nu_4\tr\(\mb \Sigma\du\mb \Sigma\)+{\tr}^2 (\mb \Sigma)+(n+1)\tr\(\mb \Sigma^2\)\right]-n^{-2}({\tr}^2\mb \Sigma_n+2n\tr\mb \Sigma_n^2),
	$$
and the covariance matrix $\Psi_2$ is given in Theorem \ref{2joint}.
\end{theorem}
This theorem shows that mean shifts appear in the asymptotic distributions of $T_1^0$ and $T_2^0$ compared with  $T_1$ and $T_2$. This result is consist with \cite{zheng2015substitution}.

\section{Simulation study}\label{sec:sim}
In this section, we conduct some simulations to verify our theoretical results. 
\subsection{Joint distribution of $T_1$ and $T_2$}
Set the parameter pairs $(p,n)$ as $(100,1000),(500,1000)$, parameter pairs $(\alpha,\beta)$ as $(0.2,0.1),(0.2,0.5),(0.5,0.5)$ and $(1,0.2)$.
For given $(p,n,\alpha,\beta)$, let $r_1,r_2,\cdots,r_p \in (0,1)$ be $p$ positive numbers, let $k=\beta p$, define $\Lambda_1=((2+r_1)n^{\alpha},(2+r_2) n^{\alpha},\cdots,(2+r_k) n^{\alpha})'\triangleq (\lambda_1^0,\lambda_2^0,\cdots,\lambda_k^0)'$ and $\Lambda_1=(2r_{k+1},\cdots,2r_p)'\triangleq (\lambda_{k+1}^0,\cdots,\lambda_p^0)'$.
We then generate a population covariance matrix $\bf \Sigma=\bf U \bf \Lambda \bf U'$ where $\bf U$ is a orthogonal matrix and $\bf \Lambda$ is a diagonal matrix whose diagonal entries are $\lambda_1^0,\cdots,\lambda_p^0.$

Now, generate a random matrix ${\bf X}=(x_{ij})_{p\times n}$ whose entries are i.i.d. and $x_{1,1}\sim Gamma(4,0.5)$. Denote $\bf B=\bf \Sigma^{1/2}\bf X\bf X'\bf \Sigma^{1/2}.$ From Theorem 2.1, we know that for large n and p, the joint distribution of $\tr {\bf B}$ and $\tr {\bf B^2}$ should close to normal and thus the distribution of $$TS\triangleq(\tr {\bf B}-\E\tr {\bf B}, \tr {\bf B^2}-\E\tr {\bf B^2})\Psi_2^{-1}(\tr {\bf B}-\E\tr {\bf B}, \tr {\bf B^2}-\E\tr {\bf B^2})'$$
should close to a standard chi-square distribution with degree of freedom 2.
The parameters $\E\tr {\bf B},\E\tr {\bf B^2}$ and $\Sigma$ can be calculated by applying Theorem 2.1.

We repeat 10000 times of generating $\bf X$ and compare the quantiles of the empirical distribution of $TS$ with the quantiles of standard chi-square distribution with degree of freedom 2. The simulation results are presented in the Figure \ref{fig1} and Figure \ref{fig2}.

\begin{figure}
  \includegraphics[width=1\textwidth,height=0.65\textwidth]{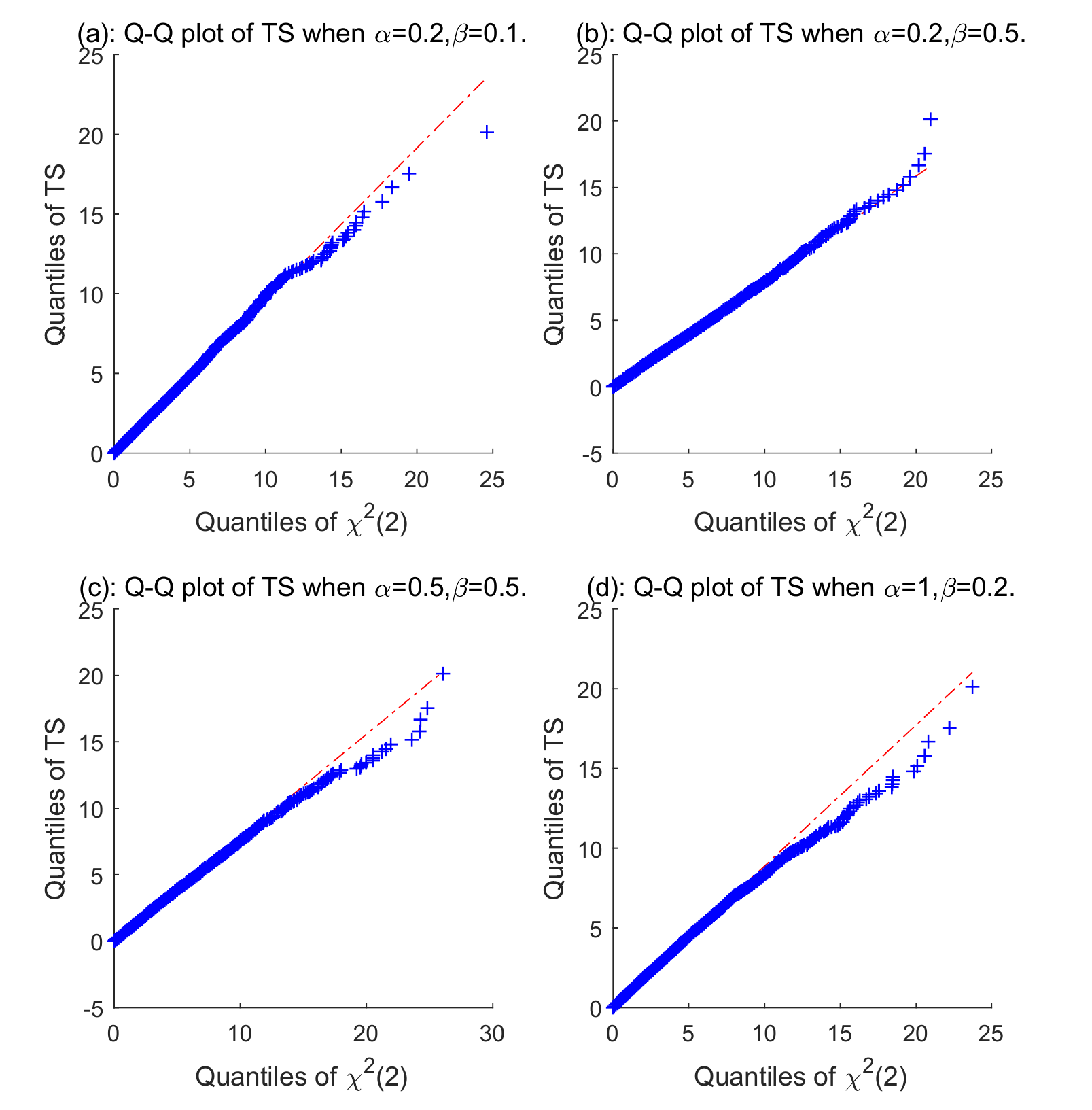}
  \caption{Q-Q plot of the empirical distribution of $TS$ vs the standard chi-square distribution with degree of freedom 2. The dimension $p=100$ and the sample size $n=1000$.}\label{fig1}	
\end{figure}

\begin{figure}
  \includegraphics[width=1\textwidth,height=0.65\textwidth]{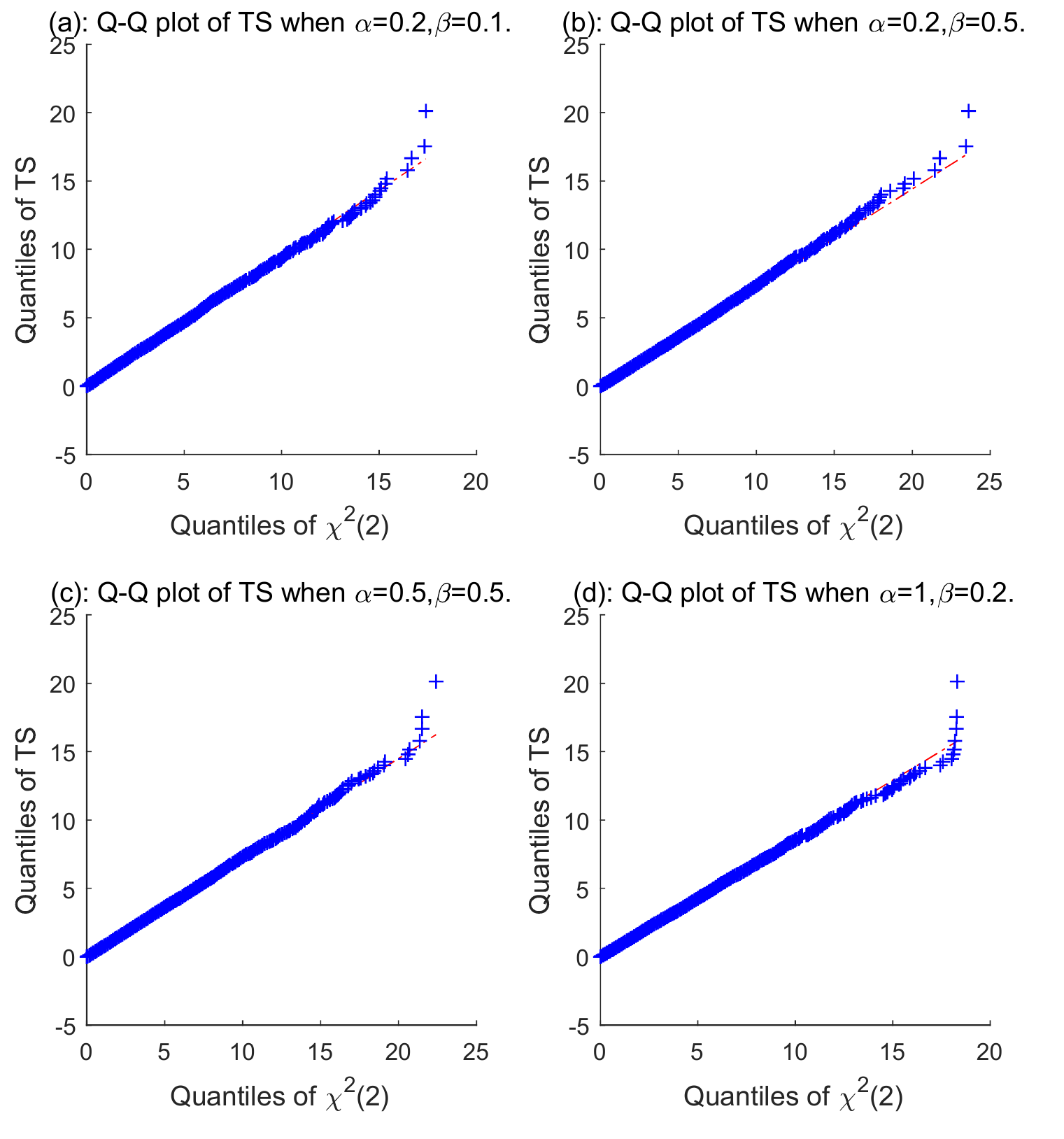}
  \caption{Q-Q plot of the empirical distribution of $TS$ vs the standard chi-square distribution with degree of freedom 2. The dimension $p=500$ and the sample size $n=1000$.}\label{fig2}	
\end{figure}

\section{Proof of Theorem \ref{2joint}}
\label{sec:proofs}
This section is to prove Theorem \ref{2joint}. The main strategy is to use the central limit theorem for martingale given in Lemma \ref{cltmar}. At first, we shall truncate the variables at a proper order when $4\alpha+\beta\leq 1$ and the $\gamma=6$. It is worth to note that the truncation step is no need when $4\alpha+\beta>1$ and the 16-th moment of the underlying distribution exists. Then we calculate the parameters in the mean vector and variance-covariance matrix of $T_1$ and $T_2$. Finally, we finish the proof by verifying the conditions in Lemma \ref{cltmar}. \subsection{Truncation, centralization and rescalling}\label{sec:trun}
In this subsection, we want to truncate the underlying variables at a proper order when $4\alpha+\beta\leq 1$ and the 6-th moment of the underlying distribution exists. 

Let $\widehat{\mb B}_n=\frac{1}{n}\mb \Sigma_n^{1/2}\widehat{\mb X}_n\widehat{\mb X}_n'\mb \Sigma_n^{1/2}$, where $\widehat{\mb X}_n=\(x_{ij}I(|x_{ij}|\leq \delta_nn^{2/\gamma})\)$ with $\delta_n$ be a sequence tend to 0 at an arbitrary slow rate. Denote $\sigma_n^2=\E|\widehat{\mb x}_{11}-\E\widehat{\mb x}_{11}|^2$ and  $\widetilde {\mb X}_n=\sigma_n^{-1}\(x_{ij}I(|x_{ij}|\leq \delta_nn^{2/\gamma})-\E x_{ij}I(|x_{ij}|\leq \delta_nn^{2/\gamma})\).$

Then we have
\begin{align*}
  {\rm P}(\widehat{\mb B_n}\neq\mb B_n,i.o.)=0.
\end{align*}

It follows that
\begin{align*}
	&{\left|\int f_{j}(x) d \widehat{G}_{n}(x)-\int f_{j}(x) d \widetilde{G}_{n}(x)\right|}{\quad \leq K_{j} \sum_{k=1}^{n}\left|\lambda_{k}^{\widehat{\mb B}_{n}}-\lambda_{k}^{\widetilde{\mb B}_{n}}\right|} \\
&{\quad \leq 2 K_{j}\left(n^{-1} \operatorname{tr} \mb \Sigma_n^{1 / 2}\left(\widehat{\mb X}_{n}-\widetilde{\mb X}_{n}\right)\left(\widehat{\mb X}_{n}-\widetilde{\mb X}_{n}\right)^{'} \mb \Sigma_n^{1 / 2}\right)^{1 / 2}\left(n\left(\lambda_{\max }^{\widehat{\mb B}_{n}}+\lambda_{\max }^{\widetilde{\mb B}_{n}}\right)\right)^{1 / 2}}.
\end{align*}
Also,
\begin{align*}
\left|1-\sigma_{n}^{2}\right|\leq& 2 \int_{\left\{\left|x_{11}\right| \geq 2 \delta_{n}n^{2/\gamma}\right\}}\left|x_{11}\right|^{2} \leq 2 \delta_{n}^{2-\gamma} n^{4/\gamma-2} \int_{\left\{\left|x_{11}\right| \geq \delta_{n} n^{2/\gamma}\right\}}\left|x_{11}\right|^{\gamma}\\
=&o\left(\delta_{n}n^{4/\gamma-2}\right),
\end{align*}
and
$$
\left|\E \widehat{x}_{11}\right|=\left|\int_{\left\{\left|x_{11}\right| \geq \delta_{n} \sqrt{n}\right\}} x_{11}\right|=o\left(\delta_{n} n^{2/\gamma-2}\right).
$$
Thus, we arrive at
\begin{align*}
&\left(n^{-1} \operatorname{tr} \mb \Sigma_n^{1 / 2}\left(\widehat{\mb X}_{n}-\widetilde{\mb X}_{n}\right)\left(\widehat{\mb X}_{n}-\widetilde{\mb X}_{n}\right)^{'} \mb \Sigma_n^{1 / 2}\right)^{1 / 2} \\
\leq &\left(n^{-1}\left(1-1 / \sigma_{n}\right)^{2} \operatorname{tr} \widehat{\mb B}_{n}\right)^{1 / 2}+\left(n^{-1}\left\|\mb \Sigma_n\right\| \sigma_n^{-2} \operatorname{tr} \E \widehat{\mb X}_{n} \E \widehat{\mb X}_{n}^{'}\right)^{1 / 2}\\
\leq &\left(\frac{\left(1-\sigma_{n}^{2}\right)^{2}}{\sigma_n^{2}\left(1+\sigma_n\right)^{2}} \frac{p}{n} \lambda_{\max }^{\widehat{B}_{n}}\right)^{1 / 2}+\left(n\left\|\mb \Sigma_n\right\|\right)^{1 / 2} \sigma_n^{-1}\left|\E \widehat{x}_{11}\right| \\
=& o\left(\delta_{n}\|\mb \Sigma_n\|^{1/2} n^{4/\gamma-2}\right)+o\left(\delta_{n}\|\mb \Sigma_n\|^{1 / 2} n^{2/\gamma-3/2}\right).
\end{align*}

We then conclude 
\begin{align*}
	&{\left|\int f_{j}(x) d \widehat{G}_{n}(x)-\int f_{j}(x) d \widetilde{G}_{n}(x)\right|}\\\notag
	=&o\left(\delta_{n}\|\mb \Sigma_n\| n^{4/\gamma-3/2}\right)+o\left(\delta_{n}\|\mb \Sigma_n\| n^{2/\gamma-1}\right)=o\left(\delta_{n}\|\mb \Sigma_n\| n^{2/\gamma-1}\).
	\end{align*}
This reveals that when $\gamma=6$ and $\alpha\leq 1/4$, we shall truncate the variables at $\delta_nn^{1/3}$ without change the asymptotically distribution of LSSs.
\subsection{Mean and Variance of $T_1$ and $T_2$ and their covariance}
After the truncation step, we focus on calculating the parameters appear in the mean vector and variance-covariance matrix.
\subsubsection{Means and Variances of $T_1$ and $T_2$ }
We now compute the mean and variance of the statistic $T_1$ and $T_2$. Firstly, we have
\begin{align*}
 \E T_1=\tr\mb \Sigma_n.
\end{align*}

Denote $\Sigma_n=(t_{i,j})$. By Lemma \ref{lm1}, we have
\begin{align}
\var T_1=n^{-1}\({\nu_4\tr\(\mb \Sigma_n^{\du2}\)+2\tr\mb \Sigma_n^2}\).
\end{align}

Now we deal with the calculation of the mean and variance of $T_2$.
Firstly, it is easy to see that
\begin{align*}
 \E T_2&={n^{-2}}\({\E\tr\(\sum_{i=1}^n\sum_{j=1}^n\mb y_i\mb y_i'\mb y_j\mb y_j'\)}\)\\
 &={n^{-2}}\({\sum_{i=1}^n\E\tr\(\mb y_i\mb y_i'\)\(\mb y_i\mb y_i'\)+\sum_{i\neq j}\tr\E\mb y_i\mb y_i'\E\mb y_j\mb y_j'}\)\\\notag
% &=\frac{n\E\(\mb y_1'\mb y_1\)^2+n(n-1)\tr(\mb \Sigma_n^2)}{n^2}
&=n^{-1}\({{\nu_4\tr\(\mb \Sigma_n^{
		\du2}\)+\(\tr \mb \Sigma_n\)^2+(n+1)\tr\(\mb \Sigma_n^2\)}}\).
\end{align*}

To obtain the variance of $T_2$, we need to calculate the second origin moment of $T_2$. We have
\begin{align}
 \E (T_2)^2&={n^{-4}}\E\tr\(\sum_{i=1}^n\sum_{j=1}^n\mb y_i\mb y_i'\mb y_j\mb y_j'\)\tr\(\sum_{i=1}^n\sum_{j=1}^n\mb y_i\mb y_i'\mb y_j\mb y_j'\)\\\notag
 &={n^{-4}}\E\tr\(\sum_{i=1}^n\sum_{j=1}^n\mb y_i\mb y_i'\mb y_j\mb y_j'\)\tr\(\sum_{k=1}^n\sum_{l=1}^n\mb y_k\mb y_k'\mb y_l\mb y_l'\)\\\notag
 &={n^{-4}}\E\(\sum_{i=1}^n\sum_{j=1}^n\sum_{k=1}^n\sum_{l=1}^n\mb y_i'\mb y_j\mb y_j'\mb y_i\mb y_k'\mb y_l\mb y_l'\mb y_k\)\\\notag
  &={n^{-4}}\E\sum_{i=1}^n\sum_{j=1}^n\sum_{k=1}^n\sum_{l=1}^n\(\mb y_i'\mb y_j\)^2\(\mb y_k'\mb y_l\)^2\\\notag
 &={n^{-4}}\({P_1+P_2+P_3+P_4+P_5}\),
\end{align}
where the term $P_m$ for $1\leq m\leq 5$ are described in the following. We will deal with the five terms one by one.

\paragraph{\bf The calculation of $P_1$:}

$P_1$ contains the summands with the indexes $i=j=k=l$. Thus, we have
\begin{align*}
 P_1&=\E\sum_{i}\(\mb y_i'\mb y_i\)^2\(\mb y_i'\mb y_i\)^2=n\E\(\mb y_1'\mb y_1\)^4.
\end{align*}
Denote below $$\xi_n={\mu_8\tr \mb \Sigma_n^4}+{\mu_6\tr \mb \Sigma_n^3\tr \mb \Sigma_n}+{\mu^2_4\(\tr \mb \Sigma_n^2\)^2},$$ where $\mu_k$ stands for the k-th origin moment of the variables.
It is easy to verify that
$$\E\(\mb y_1'\mb y_1\)^4=\Delta_1+\Delta_2+O(\xi_n),$$
where
\begin{align*}
\Delta_1
=&\E\sum_{i_1\neq i_2\neq i_3\neq i_4}t_{i_1,i_1}t_{i_2,i_2}t_{i_3,i_3}t_{i_4,i_4}x_{i_1}^2x_{i_2}^2x_{i_3}^2x_{i_4}^2\\\notag
=&\sum_{i_1,i_2,i_3,i_4}t_{i_1,i_1}t_{i_2,i_2}t_{i_3,i_3}t_{i_4,i_4}-6\sum_{i_1\neq i_2\neq i_3}t_{i_1,i_1}^2t_{i_2,i_2}t_{i_3,i_3}+O(\xi_n)\\\notag
=&\sum_{i_1,i_2,i_3,i_4}t_{i_1,i_1}t_{i_2,i_2}t_{i_3,i_3}t_{i_4,i_4}-6\sum_{i_1,i_2,i_3}t_{i_1,i_1}^2t_{i_2,i_2}t_{i_3,i_3}+O(\xi_n)\\\notag
=&(\tr\mb \Sigma_n)^4-6(\tr\mb \Sigma_n)^2\tr\(\mb \Sigma_n\du \mb \Sigma_n\)+O(\xi_n).
\end{align*}

\begin{align*}
\Delta_2
=&12\E\sum_{i_1\neq i_2\neq i_3\neq i_4}t_{i_1,i_2}^2t_{i_3,i_3}t_{i_4,i_4}x_{i_1}^2x_{i_2}^2x_{i_3}^2x_{i_4}^2+6\E\sum_{i_1\neq i_3\neq i_4}t_{i_1,i_1}^2t_{i_3,i_3}t_{i_4,i_4}x_{i_1}^4x_{i_3}^2x_{i_4}^2\\\notag
=&12\sum_{i_1\neq i_2\neq i_3\neq i_4}t_{i_1,i_2}^2t_{i_3,i_3}t_{i_4,i_4}+6\mu_4\sum_{i_1\neq i_3\neq i_4}t_{i_1,i_1}^2t_{i_3,i_3}t_{i_4,i_4}\\\notag
=&12\sum_{i_1,i_2,i_3,i_4}t_{i_1,i_2}^2t_{i_3,i_3}t_{i_4,i_4}-12\sum_{i_1,i_3,i_4}t_{i_1,i_1}^2t_{i_3,i_3}t_{i_4,i_4}\\\notag
&\quad+6\mu_4\sum_{i_1,i_3,i_4}t_{i_1,i_1}^2t_{i_3,i_3}t_{i_4,i_4}+O(\xi_n)\\\notag
=&12(\tr\mb \Sigma_n)^2\tr\mb \Sigma_n^2-12(\tr\mb \Sigma_n)^2\tr\(\mb \Sigma_n\du \mb \Sigma_n\)+6\mu_4(\tr\mb \Sigma_n)^2\tr\(\mb \Sigma_n\du \mb \Sigma_n\)+O(\xi_n).
\end{align*}

Thus, we conclude
$$P_1=n\E\(\mb y_1'\mb y_1\)^4=n(\tr\mb \Sigma_n)^4+12n(\tr\mb \Sigma_n)^2\tr\mb \Sigma_n^2+6n\nu_4(\tr\mb \Sigma_n)^2\tr\(\mb \Sigma_n\du \mb \Sigma_n\)+O(n\xi_n).$$

\paragraph{\bf The calculation of $P_2$:}

Next we consider term $P_2$, which contains the summands where there are three indexes in $i,j,k,l$ are coincident and different with the other one. We have
\begin{align*}
 P_2&=4\E\sum_{i\neq j}\(\mb y_i'\mb y_i\)^2\(\mb y_i'\mb y_j\)^2=4n_{(2)}\E\tr\mb \Sigma_n\(\mb y_1\mb y_1'\)^3\\\notag
=&4n_{(2)}\Bigg(\(\tr\mb \Sigma_n\)^2\tr\mb \Sigma_n^2+2\(\tr\mb \Sigma_n^2\)^2+4\tr\mb \Sigma_n\tr\mb \Sigma_n^3+8\tr\mb \Sigma_n^4\\\notag
&+(\mu_4-3)\(2\tr\(\mb \Sigma_n\du\mb \Sigma_n^2\)\tr \mb T+\tr\(\mb \Sigma_n^{\du2}\)\tr\mb \Sigma_n^2\)+(4\mu_4-20)\tr\(\mb \Sigma_n\mb D_{\mb \Sigma_n^2}\mb \Sigma_n\)\\\notag
&+(8\mu_4-16)\tr\(\mb \Sigma_n^2\mb D_{\mb \Sigma_n}\mb \Sigma_n\)+\mu_3^2\(4d_{\mb \Sigma_n}'\mb \Sigma_n d_{\mb \Sigma_n^2}+2d_{\mb \Sigma_n}'\mb \Sigma_n^2d_{\mb \Sigma_n}+4\mb 1'\(\mb \Sigma_n^{\du2}\du\mb \Sigma_n^2\)\mb 1\)\\\notag
&+\(\mu_6-15\mu_4-10\mu_3^2+30\)\tr\(\mb \Sigma_n^{\du2}\du\mb \Sigma_n^2\)\Bigg)\\\notag
 =&4n_{(2)}\Bigg(\(\tr\mb \Sigma_n\)^2\tr\mb \Sigma_n^2+2\(\tr\mb \Sigma_n ^2\)^2+4\tr\mb \Sigma_n\tr\mb \Sigma_n^3\\\notag
&\quad\quad+\nu_4\(2\tr\(\mb \Sigma_n\du\mb \Sigma_n^2\)\tr \mb \Sigma_n+\tr\(\mb \Sigma_n^{\du2}\)\tr\mb \Sigma_n^2\)\Bigg)+O(n^2\tr\mb \Sigma_n^4),
\end{align*}

\paragraph{\bf The calculation of $P_3$:}

$P_3$ is consist of all the summands that the four indexes are pairwise equal.
Denote by $\E_{(1)}$ the conditional expectation given $\mb x_1$. We have by lemma \ref{lm2} that
\begin{align*}
 P_3&=\E\sum_{i\neq j}\tr \(\mb y_i\mb y_i'\mb y_i\mb y_i'\)\tr \(\mb y_j \mb y_j'\mb y_j \mb y_j'\)+2\E\sum_{i\neq j}\tr \(\mb y_i\mb y_i'\mb y_j\mb y_j'\)\tr \(\mb y_i \mb y_i'\mb y_j \mb y_j'\)\\\notag
 &=n_{(2)}\(\E\(\mb y_1'\mb y_1\)^2\)^2+2n_{(2)}\E\(\mb y_1'\mb y_2\)^4\\\notag
&=n_{(2)}\(\E\(\mb y_1'\mb y_1\)^2\)^2+2n_{(2)}\(\E\E_{(1)}\(\(\mb x_2'\mb \Sigma_n\mb x_1\mb x_1'\mb \Sigma_n\mb x_2\)^2-\tr\mb \Sigma_n^2\)^2+\(\tr\mb \Sigma_n^2\)^2\)\\\notag
&=n_{(2)}\(\E\(\mb y_1'\mb y_1\)^2\)^2\\\notag
&+2n_{(2)}\(\E\(\nu_4\tr\(\mb \Sigma_n\mb x_1\mb x_1'\mb \Sigma_n\du\mb \Sigma_n\mb x_1\mb x_1'\mb \Sigma_n\)+2\tr\(\mb \Sigma_n\mb x_1\mb x_1'\mb \Sigma_n\)^2\)+\(\tr\mb \Sigma_n^2\)^2\)\\\notag
&=n_{(2)}\(\E\(\mb y_1'\mb y_1\)^2\)^2\\\notag
&+2n_{(2)}\(\E\nu_4\sum_i\(\sum_j t_{ij}x_{j,1}\)^4+2\E\tr\(\mb \Sigma_n\mb x_1\mb x_1'\mb \Sigma_n\)^2+\(\tr\mb \Sigma_n^2\)^2\)\\\notag
&=n_{(2)}\(\E\(\mb y_1'\mb y_1\)^2\)^2+2n_{(2)}\(5\nu_4\tr\(\mb \Sigma_n^2\du\mb \Sigma_n^2\)+\nu_4^2\tr\(\mb \Sigma_n^{\du2}\)^2+4\tr\mb \Sigma_n^4+3\(\tr\mb \Sigma_n^2\)^2\)\\\notag
&=n_{(2)}\(\nu_4\tr\(\mb \Sigma_n\du\mb \Sigma_n\)+2\tr\mb \Sigma_n^2+\(\tr\mb \Sigma_n\)^2\)^2\\\notag
&\quad+2n_{(2)}\(5\nu_4\tr\(\mb \Sigma_n^2\du\mb \Sigma_n^2\)+\nu_4^2\tr\(\mb \Sigma_n\du\mb \Sigma_n\)^2+4\tr\mb \Sigma_n^4+3\(\tr\mb \Sigma_n^2\)^2\)\\\notag
&=n_{(2)}\Bigg(\(\tr\mb \Sigma_n\)^4+2\nu_4\tr\(\mb \Sigma_n\du\mb \Sigma_n\)\(\tr\mb \Sigma_n\)^2+4\tr\mb \Sigma_n^2\(\tr\mb \Sigma_n\)^2+\nu_4^2\(\tr\(\mb \Sigma_n\du\mb \Sigma_n\)\)^2\\\notag
&\quad\quad+4\(\tr\mb \Sigma_n^2\)^2+4\nu_4\tr\(\mb \Sigma_n\du\mb \Sigma_n\)\tr\mb \Sigma_n^2\Bigg)+6n_{(2)}\(\tr\mb \Sigma_n^2\)^2+O(n^2\tr\mb \Sigma_n^4).
\end{align*}

\paragraph{\bf The calculation of $P_4$:}

The term $P_4$ contains those summands where there are two indexes that are coincident. In this case, we have
\begin{align*}
 P_4&=2\E\sum_{i\neq j\neq k}\tr \(\mb y_i\mb y_i'\mb y_i\mb y_i'\)\tr \(\mb y_j \mb y_j'\mb y_k \mb y_k'\)+4\E\sum_{i\neq j\neq k}\tr \(\mb y_i\mb y_i'\mb y_j\mb y_j'\)\tr \(\mb y_i \mb y_i'\mb y_k \mb y_k'\)\\\notag
 &=2n_{(3)}\tr\mb \Sigma_n^2\E\(\mb y_1'\mb y_1\)^2+4n_{(3)}\E\(\mb y_1'\mb y_2\mb y_2'\mb y_1\mb y_1'\mb y_3\mb y_3'\mb y_1\)\\\notag
&=2n_{(3)}\tr\mb \Sigma_n^2\E\(\mb y_1'\mb y_1\)^2+4n_{(3)}\(\nu_4\tr\(\mb \Sigma_n^2\du\mb \Sigma_n^2\)+2\tr\mb T^4+\(\tr\mb \Sigma_n^2\)^2\)\\\notag
&=2n_{(3)}\tr\mb \Sigma_n^2\(\nu_4\tr\(\mb \Sigma_n\du\mb \Sigma_n\)+2\tr\mb \Sigma_n^2+\(\tr\mb \Sigma_n\)^2\)\\\notag
&\quad\quad+4n_{(3)}\(\nu_4\tr\(\mb \Sigma_n^2\du\mb \Sigma_n^2\)+2\tr\mb \Sigma_n^4+\(\tr\mb \Sigma_n^2\)^2\),
\end{align*}

\paragraph{\bf The calculation of $P_5$:}

The last term $P_5$ is consist of the summands that all the four indexes are different. Here we have
\begin{align*}
 P_5&=\E\sum_{i\neq j\neq k\neq l}\tr \(\mb y_i\mb y_i'\mb y_j\mb y_j'\)\tr \(\mb y_k \mb y_k'\mb y_l \mb y_l'\)\\\notag
 &=n_{(4)}\(\tr\mb \Sigma_n^2\)^2=(n^4-6n^3+11n^2-6n)\(\tr\mb \Sigma_n^2\)^2.
\end{align*}

\paragraph{The conclusion:}
Combining the calculations above, we obtain that
\begin{align*}
&(\E T_2)^2=M_{0,1}+M_{0,2}+M_{0,3}+O(n^{-3}\xi_n+n^{-2}\tr\mb \Sigma_n^4),
\end{align*}
where
\begin{align*}
	M_{0,1}=\frac{(\tr \mb \Sigma_n)^4+n^2\(\tr\mb \Sigma_n^2\)^2+2n\(\tr\mb \Sigma_n^2\)(\tr \mb \Sigma_n)^2}{n^2},
\end{align*}
\begin{align*}
	&M_{0,2}\\
	=&\frac{2\nu_4\tr\(\mb \Sigma_n^{\du2}\)(\tr \mb \Sigma_n)^2+2\nu_4n\tr\(\mb \Sigma_n^{\du2}\)\tr \mb \Sigma_n^2+2\(\tr\mb \Sigma_n^2\)(\tr \mb \Sigma_n)^2+2n\(\tr\mb \Sigma_n^2\)^2}{n^2},
\end{align*}
and
\begin{align*}
	M_{0,3}=\frac{2\nu_4\tr\(\mb \Sigma_n^{\du2}\)\tr \mb \Sigma_n^2+\nu_4^2\(\tr\(\mb \Sigma_n^{\du2}n\)\)^2+\(\tr\mb \Sigma_n^2\)^2}{n^2}.
\end{align*}
Also, we have
\begin{align*}
&\E(L_2)^2=M_1+M_2+M_3+O(n^{-3}\xi_n+n^{-2}\tr\mb \Sigma_n^4),
\end{align*}
where
\begin{align*}
	M_1=\frac{(\tr\mb \Sigma_n)^4+2n(\tr\mb \Sigma_n^2)(\tr\mb \Sigma_n)^2+n^2(\tr\mb \Sigma_n^2)^2}{n^2},
\end{align*}
\begin{align*}
	&M_2\\
	=&\frac{2n^2\tr\mb \Sigma_n^2(\tr\mb \Sigma_n)^2+2\nu_4n^2\tr\(\mb \Sigma_n^{\du2}\)(\tr\mb \Sigma_n)^2+2\nu_4n^3\tr\(\mb \Sigma_n^{\du2}\)\tr\mb \Sigma_n^2+2n^3(\tr\mb \Sigma_n^2)^2}{n^4},
\end{align*}
and
\begin{align*}
	M_3=&n^{-4}\Bigg(8n\tr\mb \Sigma_n^2(\tr\mb \Sigma_n)^2+4\nu_4n(\tr\mb \Sigma_n)^2\tr\(\mb \Sigma_n^{\du2}\)+16n^2\tr\mb \Sigma_n\tr\mb \Sigma_n^3\\\notag
	&+5n^2\(\tr\mb \Sigma_n^2\)^2+8\nu_4n^2\tr\(\mb \Sigma_n\du\mb \Sigma_n^2\)\tr\mb \Sigma_n+4n^2\tr\(\mb \Sigma_n^{\du2}\)\tr\mb \Sigma_n^2\\\notag
	&+\nu_4^2n^2\(\tr(\mb \Sigma_n^{\du2})\)^2-2\nu_4n^2\tr\mb \Sigma_n^2\tr\(\mb \Sigma_n^{\du2}\)+4\nu_4n^3\tr\(\mb \Sigma_n^2\du\mb \Sigma_n^2\)+8n^3\tr\mb \Sigma_n^4\Bigg).
\end{align*}

Then by calculation, we finally arrive at
\begin{align}
\var(L_2)=&\Bigg[n^{-4}\Bigg(8n\tr\mb \Sigma_n^2(\tr\mb \Sigma_n)^2+4\nu_4n(\tr\mb \Sigma_n)^2\tr\(\mb \Sigma_n^{\du2}\)+16n^2\tr\mb \Sigma_n\tr\mb \Sigma_n^3\\\notag
	&+4n^2\(\tr\mb \Sigma_n^2\)^2+8\nu_4n^2\tr\(\mb \Sigma_n\du\mb \Sigma_n^2\)\tr\mb \Sigma_n\\\notag
	&+4\nu_4n^3\tr\(\mb \Sigma_n^2\du\mb \Sigma_n^2\)+8n^3\tr\mb \Sigma_n^4\Bigg)\Bigg](1+o(1)).
\end{align}
We finish the calculation of the means and variances of $T_1$ and $T_2$.
\subsubsection{The covariance of $T_1$ and $T_2$}
Now we compute the covariance of statistics $T_1$ and $T_2$.
Firstly, we have
\begin{align*}
	\E L_1L_2&=n^{-3}\E\(\(\tr\sum_{i=1}^n \mb y_i'\mb y_i\)\tr\(\sum_{j=1}^n\mb y_j'\mb y_j\)^2\)\\\notag
	&=n^{-3}\E\sum_{i,j,k=1}^n\mb y_i'\mb y_i\mb y_j'\mb y_k\mb y_k'\mb y_j:=n^{-3}\(P_1+P_2+P_3\).
\end{align*}
We now process the steps one by one.
\paragraph{\bf Calculation of  $P_1$:}
$P_1$ contains the summands where the three indexes are equal. We obtain from Lemma \ref{lm3} that
\begin{align*}
	P_1=&n\E\(\mb y_1'\mb y_1\)^3=n\E\(\mb x_1'\mb \Sigma_n\mb x_1\)^3\\\notag
	=&n\(\tr\mb \Sigma_n\)^2\tr\mb \Sigma_n +2n\tr\(\mb \Sigma_n\)^2\tr \mb \Sigma_n+4n\tr\(\mb \Sigma_n\)\tr\mb \Sigma_n^2+8n\tr\(\mb \Sigma_n^3\)\\\notag
&+3(\mu_4-3)n\(\tr\(\mb \Sigma_n\du\mb \Sigma_n\)\tr \mb \Sigma_n\)+(4\mu_4-20)n\tr\(\mb \Sigma_n\mb D_{\mb \Sigma_n}\mb \Sigma_n\)\\\notag
&+(8\mu_4-16)n\tr\(\mb \Sigma_n\mb D_{\mb \Sigma_n}\mb \Sigma_n\)\\
&+\mu_3^2n\(4d_{\mb \Sigma_n}'\mb \Sigma_n d_{\mb \Sigma_n}+2d_{\mb \Sigma_n}'\mb \Sigma_n d_{\mb \Sigma_n}+4\mb 1'\(\mb \Sigma_n\du \mb \Sigma_n\du\mb \Sigma_n\)\mb 1\)\\\notag
&+\(\mu_6-15\mu_4-10\mu_3^2+30\)n\tr\(\mb T\du \mb T\du\mb \Sigma_n\)\\\notag
=&n\(\tr\mb \Sigma_n\)^3 +6n\tr\mb \Sigma_n^2\tr \mb \Sigma_n+3\nu_4n\(\tr\(\mb \Sigma_n\du\mb \Sigma_n\)\tr \mb \Sigma_n\)+O(n\tr\mb \Sigma_n^3).
\end{align*}

\paragraph{\bf Calculation of  $P_2$:}
$P_2$ contains the summands where there are two of three indexes are equal. It is easy to see that
\begin{align*}
	P_2=&2\E\sum_{i\neq j}^n\mb y_i'\mb y_i\mb y_i'\mb y_j\mb y_j'\mb y_i+\E\sum_{i\neq j}^n\mb y_i'\mb y_i\mb y_j'\mb y_j\mb y_j'\mb y_j\\\notag
	=&2n_{(2)}\E\(\mb x_1'\mb \Sigma_n\mb x_1\)\(\mb x_1'\mb \Sigma_n^2\mb x_1\)+n_{(2)}\(\tr\mb \Sigma_n\)\E\(\mb x_1'\mb \Sigma_n\mb x_1\)^2\\\notag
	=&2n_{(2)}\(\nu_4\tr\(\mb \Sigma_n\du\mb \Sigma_n^2\)+2\tr\mb \Sigma_n^3+\tr\mb \Sigma_n\tr\mb \Sigma_n^2\)\\\notag
	&+n_{(2)}\tr\mb \Sigma_n\(\nu_4\tr\(\mb \Sigma_n\du\mb \Sigma_n\)+2\tr\mb \Sigma_n^2+\(\tr\mb \Sigma_n\)^2\)\\\notag
	=&n_{(2)}\(2\nu_4\tr\(\mb \Sigma_n\du\mb \Sigma_n^2\)+4\tr\mb \Sigma_n^3+4\tr\mb \Sigma_n\tr\mb \Sigma_n^2+\nu_4\tr\(\mb \Sigma_n^{\du2}\)\tr\mb \Sigma_n+\(\tr\mb \Sigma_n\)^3\)
\end{align*}

\paragraph{\bf Calculation of  $P_3$:}
$P_3$ contains the summands where the three indexes are all different. We have
\begin{align*}
	P_3=n_{(3)}\tr\mb \Sigma_n\tr\mb \Sigma_n^2
\end{align*}

Combining the calculates above, we obtain that
\begin{align*}
	&\E T_1T_2\\
	=&n^{-3}\Bigg(n\(\tr\mb \Sigma_n\)^3 +6n\tr\mb \Sigma_n^2\tr \mb \Sigma_n+3\nu_4n\(\tr\(\mb \Sigma_n^{\du2}\)\tr \mb \Sigma_n\)\\\notag
	&+n_{(2)}\(2\nu_4\tr\(\mb \Sigma_n\du\mb \Sigma_n^2\)+4\tr\mb \Sigma_n^3+4\tr\mb \Sigma_n\tr\mb \Sigma_n^2+\nu_4\tr\(\mb \Sigma_n^{\du2}\)\tr\mb \Sigma_n+\(\tr\mb \Sigma_n\)^3\)\\\notag
	&+(n^3-3n^2+2n)(\tr\mb \Sigma_n)(\tr\mb \Sigma_n^2)+O(n\tr\mb \Sigma_n^3)\Bigg)\\\notag
	=&\Bigg[n^{-3}\Bigg(n^2\(\tr\mb \Sigma_n\)^3 +(n^3-3n^2+8n)\tr\mb \Sigma_n^2\tr \mb \Sigma_n+3\nu_4n\(\tr\(\mb \Sigma_n^{\du2}\)\tr \mb \Sigma_n\)\\\notag
	&+n_{(2)}\(2\nu_4\tr\(\mb \Sigma_n\du\mb \Sigma_n^2\)+4\tr\mb \Sigma_n^3+4\tr\mb \Sigma_n\tr\mb \Sigma_n^2+\nu_4\tr\(\mb \Sigma_n^{\du2}\)\tr\mb \Sigma_n\)\Bigg)\Bigg](1+o(1)).
\end{align*}
Notice that
$$\E T_1\E T_2=\frac{\nu_4\tr\(\mb \Sigma_n^{\du2}\)\tr\mb \Sigma_n+\(\tr\mb \Sigma_n\)^3+n\tr\mb \Sigma_n^2\tr\mb \Sigma_n+\tr\mb \Sigma_n^2\tr\mb \Sigma_n}{n}.$$
Thus, we conclude that
\begin{align}
&{\rm Cov}(T_1,T_2)\\\notag
=&\Bigg[n^{-2}\Bigg(4\tr\mb \Sigma_n^2\tr \mb \Sigma_n+2\nu_4\tr\(\mb \Sigma_n^{\du2}\)\tr \mb \Sigma_n+2\nu_4n\tr\(\mb \Sigma_n\du\mb \Sigma_n^2\)+4n\tr\mb \Sigma_n^3\Bigg)\Bigg](1+o(1)).
\end{align}
Then we complete this part.

\subsection{Complete the proof of Theorem \ref{2joint}}
The main task is to prove that for any $a$ and $b$, $T(a,b)=aT_1+bT_2$ is asymptotically normal. To this end, we shall apply Lemma \ref{cltmar}.
\subsubsection{Martingale difference decomposition}
We first decompose $T(a,b)$ into sum of Martingale difference sequence. Let $\E_k(\cdot)$ denote the condition expectation given $\{\mb x_1,\mb x_2,\cdots,\mb x_k\}$.
We have
\begin{align*}
  \E_k(T_1)=&\frac{1}{n}\E_k\tr\sum_{i=1}^n\mb y_i\mb y_i'=\frac{1}{n}\sum_{i=1}^k\tr\mb y_i\mb y_i'+\frac{n-k}{n}\tr\mb \Sigma_n,
\end{align*}
and
\begin{align*}
  \E_k(T_2)=&\frac{1}{n^2}\E_k\tr\sum_{i=1}^n\sum_{j=1}^n\mb y_i\mb y_i'\mb y_j\mb y_j'\\
      =&\frac{1}{n^2}\E_k\tr\(\sum_{i=1}^k\sum_{j=1}^k+\sum_{i=k+1}^n\sum_{j=k+1}^n+2\sum_{i=1}^k\sum_{j=k+1}^n\)\tr\mb y_i\mb y_i'\mb y_j\mb y_j'\\\notag
=&\frac{1}{n^2}\sum_{i=1}^k\sum_{j=1}^k\tr\mb y_i\mb y_i'\mb y_j\mb y_j'+\frac{n-k}{n^2}\E\(\mb y_1'\mb y_1\)^2\\\notag
&+\frac{(n-k)(n-k-1)}{n^2}\tr\(\mb \Sigma_n^2\)+\frac{2(n-k)}{n^2}\sum_{i=1}^k\tr\mb y_i\mb y_i'\mb \Sigma_n.
\end{align*}

Then we obtain
\begin{align*}
  D_k^{(1)}=&\(\E_k-\E_{k-1}\)T_1=\frac{\mb y_k'\mb y_k}{n}-\frac{\tr \mb \Sigma_n}{n}:=D_{k,1},
\end{align*}
and
\begin{align*}
  D_k^{(2)}=&\(\E_k-\E_{k-1}\)T_2\\\notag
  =&\frac{2}{n^2}\sum_{i=1}^{k-1}\(\tr\(\mb y_i\mb y_i'\mb y_k\mb y_k'\)-\tr\(\mb y_i\mb y_i'\mb \Sigma_n\)\)+\frac{2(n-k)}{n^2}\(\tr\(\mb y_k\mb y_k'\mb \Sigma_n\)-\tr\(\mb \Sigma_n^2\)\)\\\notag
  &+\frac{1}{n^2}\(\tr\(\(\mb y_k\mb y_k'\)^2\)-\E\(\mb y_1\mb y_1'\)^2\)\\\notag
  =&\frac{2}{n^2}\sum_{i=1}^{k-1}\(\mb x_k'\mb \Sigma_n^{1/2}\mb y_i'\mb y_i\mb \Sigma_n^{1/2}\mb x_k-\tr\(\mb y_i\mb y_i'\mb \Sigma_n\)\)+\frac{2(n-k)}{n^2}\(\mb x_k'\mb \Sigma_n^2\mb x_k-\tr\(\mb \Sigma_n^2\)\)\\\notag
  &+\frac{1}{n^2}\(\tr\(\(\mb y_k\mb y_k'\)^2\)-\E\(\mb y_1\mb y_1'\)^2\):=D_{k,2}+D_{k,3}+D_{k,4}.
\end{align*}

Now we obtain
$$D_k=(\E_k-\E_{k-1})T(a,b)=aD_{k,1}+b\(D_{k,2}+D_{k,3}+D_{k,4}\).$$
\subsubsection{The verification of Lindberg condition}
This subsection is to verify the Lindberg condition. For $D_{k,1}$, we have
\begin{align*}
	\sum_{k=1}^n\E\left|D_{k,1}\right|^4&\leq Cn^{-3}\E\left|\mb x_1'\mb \Sigma_n \mb x_1-\tr \mb \Sigma_n\right|^4\leq Cn^{-3}\(\mu_8\tr\mb \Sigma_n^4+\mu_4^2\tr^2\mb \Sigma_n^2\)\\
	&=O(\frac{\tr\mb \Sigma_n^4}{n^2}+\frac{\tr^2\mb \Sigma_n^2}{n^3}).
\end{align*}
For $D_{k,2}$, we have
\begin{align*}
	&\sum_{k=1}^n\E\left|D_{k,2}\right|^4=\sum_{k=1}^n\E\left|\frac{2}{n^2}\sum_{i=1}^{k-1}\(\mb x_k'\mb \Sigma_n^{1/2}\mb y_i'\mb y_i\mb \Sigma_n^{1/2}\mb x_k-\tr\(\mb y_i\mb y_i'\mb \Sigma_n\)\)\right|^4\\\notag
	\leq& Cn^{-3}\E\left|\mb x_1'\mb \Sigma_n^{1/2}\mb y_2'\mb y_2\mb \Sigma_n^{1/2}\mb x_1-\tr\(\mb y_2\mb y_2'\mb \Sigma_n\)\right|^4\\\notag
	\leq& Cn^{-3}\(\mu_8\E\(\mb x_2'\mb \Sigma_n^2\mb x_2-\tr\mb \Sigma_n^2\)^4+\mu_4^2\(\E\(\mb x_2'\mb \Sigma_n^2\mb x_2-\tr\mb \Sigma_n^2\)^2\)^2\)\\\notag
	\leq&Cn^{-3}\(\mu_8\(\mu_8\tr\mb \Sigma_n^8+\mu_4^2\tr^2\mb \Sigma_n^4\)+\mu_4^2\(\mu_4\tr\mb \Sigma_n^4\)^2\)=o(1)+O(\frac{\tr^2\mb \Sigma_n^4}{n^3}),
\end{align*}
where the $o(1)$ is to control the order when $4\alpha+\beta\leq 1$ and $O(\frac{\tr^2\mb \Sigma_n^4}{n^3})$ is to control the order when $4\alpha+\beta>1.$

Then for $D_{k,3}$, we shall obtain
\begin{align*}
	&\sum_{k=1}^n\E\left|D_{k,3}\right|^4=\sum_{k=1}^n\E\left|\frac{2(n-k)}{n^2}\(\mb x_k'\mb \Sigma_n^2\mb x_k-\tr\(\mb \Sigma_n^2\)\)\right|^4\\\notag
	\leq& Cn^{-3}\E\left|\mb x_1'\mb \Sigma_n^2\mb x_1-\tr\(\mb \Sigma_n^2\)\right|^4\leq Cn^{-3}\(\mu_8\tr\mb \Sigma_n^8+\mu_4^2\tr^2\mb \Sigma_n^4\)=o(\frac{\tr^2\mb \Sigma_n^4}{n^2}),
\end{align*}
For $D_{k,4}$, we have from calculation that
\begin{align*}
	&\sum_{k=1}^n\E\left|D_{k,4}\right|^4\\
=&\sum_{k=1}^n\E\left|\frac{1}{n^2}\(\tr\(\(\mb y_k\mb y_k'\)^2\)-\E\(\mb y_1\mb y_1'\)^2\)\right|^4\\\notag
	\leq& Cn^{-7}\E\left|\(\mb y_1'\mb y_1		\)^2-\E\(\mb y_1\mb y_1'\)^2\right|^4\\\notag
	\leq& Cn^{-7}\Bigg(\sum_{1\leq i_1,\cdots,i_6\leq 8,i_1+\cdots+i_6=8}\prod_{\tau=1}^6\mu_{\max\{2i_{\tau}-6,6\}}\prod_{\tau=1}^6\tr\mb \Sigma_n^{i_{\tau}}\Bigg)\\\notag
=&o(1)+o(\frac{\tr^2\mb \Sigma_n^4}{n^2}+\frac{\tr^4\mb \Sigma_n^2}{n^4}+\frac{\tr^2\mb \Sigma_n\tr^2\mb \Sigma_n^3}{n^4}).
\end{align*}
Thus, we arrive at
\begin{align}
  \sum_{k=1}^n\E|D_{k}|^4=o(1+\frac{\tr^2\mb \Sigma_n^4}{n^2}+\frac{\tr^4\mb \Sigma_n^2}{n^4}+\frac{\tr^2\mb \Sigma_n\tr^2\mb \Sigma_n^3}{n^4}).
\end{align}
That is to say, we have $\sum_{k=1}^n\E|D_{k}|^4=o({\rm Var}^2\(T(a,b)\)).$
\subsubsection{Complete the proof of CLT}
The remaining task is to prove that $${\rm Var} \sum_{k=1}^n \E_{k-1}D_k^2=o({\rm Var}^2\(T(a,b)\)).$$ To this end, denote $\mb C_i=\mb \Sigma_n^{1/2}\mb y_i\mb y_i'\mb \Sigma_n^{1/2}$, by Lemma \ref{lm1} we have
\begin{align*}
  &\E_{k-1}D_{k,1}^2=\E_{k-1}\left|\frac{\mb y_k'\mb y_k}{n}-\frac{\tr \mb \Sigma_n}{n}\right|^2=n^{-2}\(\nu_4\tr(\mb \Sigma_n\du \mb \Sigma_n)+2\tr(\mb \Sigma_n^2)\).
\end{align*}
Also, we shall obtain
\begin{align*}
  \E_{k-1}D_{k,2}^2&=\frac{4}{n^4}\sum_{i=1}^{k-1}\sum_{j=1}^{k-1}\E_{k-1}\(\mb x_k'\mb C_i\mb x_k-\tr\mb C_i\)\(\mb x_k'\mb C_j\mb x_k-\tr\mb C_j\)\\\notag
  &=\frac{4}{n^4}\sum_{i=1}^{k-1}\sum_{j=1}^{k-1}\(\nu_4\tr\mb C_i\du \mb C_j+2\tr\mb C_i\mb C_j\),
\end{align*}
and
\begin{align*}
  \E_{k-1}D_{k,3}^2&=\frac{4(n-k)^2}{n^4}\E_{k-1}\(\mb x_k'\mb \Sigma_n^2\mb x_k-\tr\(\mb \Sigma_n^2\)\)^2\\
  &=\frac{4(n-k)^2}{n^4}\(\nu_4\tr\(\mb \Sigma_n^2 \du \mb \Sigma_n^2\)+2\tr\(\mb \Sigma_n^4\)\).
\end{align*}
We also have
\begin{align*}
  &\E_{k-1}D_{k,2}D_{k,3}\\
  =&\frac{4(n-k)}{n^4}\sum_{i=1}^{k-1}\(\mb x_k'\mb \Sigma_n^{1/2}\mb y_i\mb y_i'\mb \Sigma_n^{1/2}\mb x_k-\tr\(\mb y_i\mb y_i'\mb \Sigma_n\)\)\(\mb x_k'\mb \Sigma_n^2\mb x_k-\tr\(\mb \Sigma_n^2\)\)\\\notag
  =&\frac{4(n-k)}{n^4}\sum_{i=1}^{k-1}\(\nu_4\tr\(\mb C_i \du \mb \Sigma_n^2\)+2\tr\(\mb C_i \mb \Sigma_n^2\)\),
\end{align*}
and
\begin{align*}
  &\E_{k-1}D_{k,4}^2=\frac{1}{n^4}\(\E\(\mb y_1'\mb y_1\)^4-\(\E\(\mb y_1'\mb y_1\)^2\)^2\).
\end{align*}
What is more, one have
\begin{align*}
  &\E_{k-1}D_{k,2}D_{k,4}=\frac{2}{n^4}\sum_{i=1}^{k-1}\(\E_{k-1}\tr\(\mb y_i\mb y_i'\(\mb y_k \mb y_k'\)^3\)-\E\(\mb y_1'\mb y_1\)^2\tr\(\mb y_i\mb y_i'\mb \Sigma_n\)\),
\end{align*}
and
\begin{align*}
  &\E_{k-1}D_{k,3}D_{k,4}=\frac{2(n-k)}{n^4}\(\E_{k-1}\tr\(\mb \Sigma_n\(\mb y_k\mb y_k'\)^3\)-\E\(\mb y_1'\mb y_1\)^2\tr\mb \Sigma_n^2\).
\end{align*}
Then, we can get that
\begin{align*}
  &\var\(\sum_{k=1}^n\E_{k-1}D_{k,2}^2\)\\\notag
  =&\var\(\sum_{k=1}^n\frac{4}{n^4}\sum_{i=1}^{k-1}\sum_{j=1}^{k-1}\(\nu_4\tr\mb C_i\du \mb C_j+2\tr\mb C_i\mb C_j\)\)\\\notag
  =&16n^{-8}\E\(\sum_{k=1}^n\sum_{i=1}^{k-1}\sum_{j=1}^{k-1}\nu_4\tr\mb C_i\du \mb C_j+2\tr\mb C_i\mb C_j-\E\(\nu_4\tr\mb C_i\du \mb C_j+2\tr\mb C_i\mb C_j\)\)^2\\\notag
  \leq &Cn^{-6}\E\(\sum_{i=1}^{k-1}\sum_{j=1}^{k-1}\nu_4\tr\mb C_i\du \mb C_j+2\tr\mb C_i\mb C_j-\E\(\nu_4\tr\mb C_i\du \mb C_j+2\tr\mb C_i\mb C_j\)\)^2\\\notag
  \leq &Cn^{-6}\E\sum_{i_1,i_2=1}^{k-1}\sum_{j=1}^{k-1}\(\nu_4\tr\mb C_{i_1}\du \mb C_j+2\tr\mb C_{i_1}\mb C_j-\E\(\nu_4\tr\mb C_{i_1}\du \mb C_j+2\tr\mb C_{i_1}\mb C_j\)\)\\\notag
  &\quad\quad\quad\quad \times \(\nu_4\tr\mb C_{i_2}\du \mb C_j+2\tr\mb C_{i_2}\mb C_j-\E\(\nu_4\tr\mb C_{i_2}\du \mb C_j+2\tr\mb C_{i_2}\mb C_j\)\)\\\notag
  \leq &Cn^{-3}\Bigg(\E^{1/4}\left|\mb x_1'\mb \Sigma_n^{1/2}\mb y_2\mb y_2'\mb \Sigma_n^{1/2}\mb x_1-\tr\(\mb y_2\mb y_2'\mb \Sigma_n\)\right|^4\\
  &\quad\quad\times\E^{1/4}\left|\mb x_1'\mb \Sigma_n^{1/2}\mb y_3\mb y_3'\mb \Sigma_n^{1/2}\mb x_1-\tr\(\mb y_3\mb y_3'\mb \Sigma_n\)\right|^4\\\notag
  &\quad\quad\times \E^{1/2}\left|\mb x_1'\mb \Sigma_n^{1/2}\mb y_4\mb y_4'\mb \Sigma_n^{1/2}\mb x_1-\tr\(\mb y_4\mb y_4'\mb \Sigma_n\)\right|^4\Bigg)\\
=&o(1+\frac{\tr^2\mb \Sigma_n^4}{n^2}+\frac{\tr^4\mb \Sigma_n^2}{n^4}+\frac{\tr^2\mb \Sigma_n\tr^2\mb \Sigma_n^3}{n^4}).
\end{align*}
The estimate of the other terms are the same or simpler thus omitted.

Then we are done.

\section{Proof of Theorem \ref{2jointm}}\label{sec:proof2}
To prove Theorem \ref{2jointm}, we only need to  investigate the effect of the sample mean. Recall that $\bar {\mb y}=\frac1n\sum_{i}\mb y_i,$ we have
$$\mb B_n^0=\frac1n\sum_{i=1}^n\(\mb y_i-\frac1n\sum_{j=1}^n\mb y_j\)\(\mb y_i-\frac1n\sum_{j=1}^n\mb y_j\)'=\frac1n\sum_{i=1}^n\mb y_i\mb y_i'-\bar {\mb y}\bar {\mb y}'=\mb B_n-\bar {\mb y}\bar {\mb y}'.$$
Also, we see that
$T_1^0=T_1-\bar {\mb y}'\bar {\mb y},$ and thus
$\E \bar {\mb y}'\bar {\mb y}=\frac{\tr\mb \Sigma_n}{n}.$ Then, we have
\begin{align*}
\E\(\bar {\mb y}'\bar {\mb y}\)^2
&=\frac{1}{n^4}\E\sum_{i_1,\cdots,i_4}\mb y_{i_1}'\mb y_{i_2}\mb y_{i_3}'\mb y_{i_4}\\\notag
&=\frac{\nu_4\tr(\mb \Sigma_n^{\du2})+\tr \mb \Sigma_n^2+\(\tr\mb \Sigma_n\)^2}{n^3}+\frac{n_{(2)}\(\tr\mb \Sigma_n\)^2}{n^4}+\frac{n_{(2)}\(2\tr\mb \Sigma_n^2\)}{n^4},
\end{align*}
This implies
${\rm Var \(\bar {\mb y}'\bar {\mb y}\)}=\frac{2\tr\mb \Sigma_n^2}{n^2}(1+o(1)).$

Next, we have
\begin{align*}
  T_2^0&=\tr\(\mb B_n^0\)^2=T_2+\(\bar {\mb y}'\bar {\mb y}\)^2-2\bar {\mb y}'\mb B_n\bar {\mb y}.
\end{align*}
From calculation, we obtain
\begin{align*}
  \E \bar {\mb y}'\mb B_n\bar {\mb y}%=&\frac{1}{n^3}\E\sum_{i_1,i_2,i_3}\mb y_{i_1}'\mb y_{i_2}\mb y_{i_2}'\mb y_{i_3}\\\notag
  =&\frac{\nu_4\tr(\mb \Sigma_n^{\du2})+\tr \mb \Sigma_n^2+\(\tr\mb \Sigma_n\)^2}{n^2}+\frac{n_{(2)}\tr\mb \Sigma_n^2}{n^3}\\\notag
  =&\frac{\nu_4\tr(\mb \Sigma_n^{\du2})+n\tr \mb \Sigma_n^2+\(\tr\mb \Sigma_n\)^2}{n^2}.
\end{align*}
And
\begin{align*}\label{exmple}
  &{\rm Var} \(\bar {\mb y}'\mb B_n\bar {\mb y}\)\\\notag
  =&\E\(\bar {\mb y}'\mb B_n\bar {\mb y}-\E\bar {\mb y}'\mb B_n\bar {\mb y}\)^2 \\\notag
  =&\frac{\E\(\sum_{i_1=1}^n\(\(\mb y_{i_1}'\mb y_{i_1}\)^2-\E\(\mb y_1'\mb y_1\)^2\)+\sum_{i_1\neq i_2}\(\mb y_{i_1}'\mb y_{i_2}\mb y_{i_2}'\mb y_{i_1}-\tr\mb \Sigma_n^2\)\)^2}{n^6}\\\notag
  \leq &{C n^{-6}}\Bigg[n\E\(\(\mb y_{1}'\mb y_{1}\)^2-\E\(\mb y_1'\mb y_1\)^2\)^2+\(n^2\tr^2\mb \Sigma_n^2+n^3\tr\mb \Sigma_n^4+\nu_4n^3\tr\(\mb \Sigma_n^2\du\mb \Sigma_n^2\)\)\Bigg]\\\notag
  \leq&{C n^{-6}}\Bigg[n\(\tr\mb \Sigma_n\)^4+\(n^2\tr^2\mb \Sigma_n^2+n^3\tr\mb \Sigma_n^4+\nu_4n^3\tr\(\mb \Sigma_n^2\du\mb \Sigma_n^2\)\)\Bigg]\\\notag
  \leq &C\(\frac{\tr^2\mb \Sigma_n^2}{n^4}+\frac{\tr\mb \Sigma_n^4}{n^3}\).
\end{align*}
Thus we conclude that
\begin{align*}
  \E (T_2^0-T_2)=\(\frac{-\tr^2\mb \Sigma_n-2n\tr\mb \Sigma_n^2}{n^2}\)(1+o(1)).
\end{align*}
This completes the proof of this theorem.
\section{proof of Theorem \ref{kjoint} by moment method}\label{sec:proof3}
The proof of this theorem is based on the moment method. We note that the original idea appears in \cite{jonsson1982some}. However, the population covariance matrix in \cite{jonsson1982some} is assumed to be identity. As will be seen from the proof below, the extension to the non-null population covariance matrix case is non-trivial and much more efforts have to be made.  

\subsection{Some primary definitions and lemmas}
At first, we note that from the truncation step presented in \ref{sec:trun}, we shall truncate the variable at $\eta_n\sqrt{n}$ where $\eta_n$ convergence to 0 since $\|\mb \Sigma\|=O(\log n)$ and $\gamma >4$.
Recall that \begin{align}
	&T_k=\tr\mb B_n^k=n^{-k}\tr\(\mb X_n'\mb \Sigma\mb X_n\)^k\\\notag
	=&n^{-k}\sum_{i_1,i_2,\cdots,i_{2k},j_1,\cdots,j_k}x_{i_1,j_1}t_{i_1,i_2}x_{i_2,j_2}x_{i_3,j_2}t_{i_3,i_4}x_{i_4,j_3}\cdots x_{i_{2k-1},j_{k}}t_{i_{2k-1},i_{2k}}x_{i_{2k},j_1}.
\end{align}

Denote $\sigma_{\phi,\psi}=\E\(T_{\phi}-\E T_{\phi}\)\(T_{\psi}-\E T_{\psi}\).$
For two sequences $\(i_1,i_2,\cdots,i_{2k}\)$  and  $\(j_1,j_2,\cdots,j_k\)$, $1\leq i_l\leq p$ for all $1\leq l\leq 2k$, $1\leq j_q\leq n$ for all $1\leq q\leq k$, we shall define a $\mb{Q}_k$-$\mb{graph}$ in the following way. Draw two parallel lines, referred to as the $I$-line (upper) and $J$-line (lower), plot $i_1,\cdots,i_{2k}$ on the $I$-line and $j_1,\cdots,j_k$ on the $J$-line, called the $I$-vertices and $J$-vertices respectively. Then draw $k$ down edges from $i_{2l-1}$ to $j_l$, $k$ up edges from $j_l$ to $i_{2l}$ (the $k$ down edges and $k$ up edges are called vertical edges), and $k$ horizontal edges from $i_{2l-1}$ to $i_{2l}$ . Define $V$ to be the set of distinct numbers of $i_1,\cdots,i_{2k}$ and $j_1,\cdots,j_{k}$, $E=\{e_{dl},e_{ul},e_{hl},l=1,\cdots,k\}$, and the function $F(e_{dl})=(i_{2l-1},j_l)$, $F(e_{ul})=(j_{l},i_{2l})$, $F(e_{hl})=(i_{2l},i_{2l-1})$, then $(V,E,F)$ is called a $\mb{Q}_k$-$\mb{graph}$, denoted as $\mb{Q}_k(V,E,F)$. The union of $\mb{Q}_k$ graphs is called a $\mb Q$-$\mb{graph}$ and every $\mb{Q}_k$ graph is called a {\bf basic graph}.

Let $Q$ be a $\mb{Q}$-graph. The sub-graph of $Q$ that containing all $I$-vertices and all horizontal edges is called the $\mb{roof}$ of $Q$, denoted as $R(Q)$. If we remove all horizontal edges from $Q$ and glue all $I$-vertices that connected through horizontal edges, we get the $\mb{pillar}$ of $Q$, denoted as $P(Q)$. If we glue all the coincident vertical edges in $Q$, then we get the $\mb{nerve}$ of $Q$, denoted as $N(Q)$. Of course the nerve is a connected graph.
Two $\mb Q$-graphs are said to be ${\bf isomorphic}$ if one can be converted to the other by a permutation of $(1,\cdots,p)$ and a permutation of $(1,\cdots,n)$. Thus, all the $\mb Q$-graphs are classified into isomorphic groups. We shall choose one from each group as the ${\bf standard}$ $\mb Q$-graph of that group.
A vertical edge in $Q$ is non-coincident with any other edges is called a ${\bf Single}$ edge. A vertical edge in $Q$ that coincident with one and only one edge (ignore the direction) is called a ${\bf double}$ edge. A vertical edge in $Q$ that coincident with more than one edges (its multiplicity is at least 3 ignore the direction) is called a ${\bf multiple}$ edge.

Now, we shall link $T_k$ with $\mb{Q}_k$-{graph}, where the vertical edges correspond to the random variables while the horizontal edges correspond to the entries of the population covariance matrix $\mb \Sigma$.

We now introduce some fundamental lemmas form graph-associated matrix theory.
\begin{lemma}[Theorem A.31 in \citep{BSbook}]\label{mmlm1}
	Suppose that $\mb G=\(\mb V,\mb E, \mb F\)$ is a two-edge connected graph with $t$ vertices and $k$ edges. Each vertex $i$ corresponds to an integer $m_i \geq 2$ and each edge $e_j$ corresponds to a matrix $\mb \Sigma^{(j)}=\(t_{\alpha,\beta}^{(j)}\),\ j=1,\cdots,k$, with consistent dimensions, that is, if $F(e_j)=(f_i(e_j),f_e(e_j))=(g,h),$ then the matrix $\mb \Sigma^{\(j\)}$ has dimensions $m_g\times m_h$. Define $\mb v=(v_1,v_2,\cdots,v_t)$ and
	\begin{align}
		T'=\sum_{\mb v}\prod_{j=1}^kt_{v_{f_i(e_j)},v_{f_e(e_j)}}^{(j)},
	\end{align}
	where the summation $\sum_{\mb v}$ is taken for $v_i=1,2,\cdots, m_i, \ i=1,2,\cdots,t.$ Then for any $i\leq t$,
	we have
	$$|T'|\leq \max_i{m_i}\prod_{j=1}^k\|\mb \Sigma^{(j)}\|.$$
\end{lemma}

\begin{lemma}\label{roof}
	Suppose that $\mb G=\(\mb V,\mb E, \mb F\)$ is a graph with $t$ vertices and $k$ edges. Each vertex $i$ corresponds to an integer $m_i \geq 2$ and each edge $e_j$ corresponds to a matrix $\mb \Sigma^{(j)}=\(t_{\alpha,\beta}^{(j)}\),\ j=1,\cdots,k$, with consistent dimensions. Define $\mb v=(v_1,v_2,\cdots,v_t)$ and
	\begin{align}
		T'=\sum_{\mb v}\prod_{j=1}^kt_{v_{f_i(e_j)},v_{f_e(e_j)}}^{(j)},
	\end{align}
	where the summation $\sum_{\mb v}$ is taken for $v_i=1,2,\cdots, m_i, \ i=1,2,\cdots,t$ and subject to the restriction that
	the indicators are different. Then for any $i\leq t$,
	we have
	$$|T'|\leq C_k\(\max_i{m_i}\)^{\max{\{\theta/2,1\}}}\prod_{j=1}^k\|\mb \Sigma^{(j)}\|,$$
	where $\theta$ is the number of vertices of $\mb G=\(\mb V,\mb E, \mb F\)$ that have odd degree and $C_k$ is a constant depend on $k$ only.
\end{lemma}

\begin{proof}
	Since the number of vertices that have odd degree must be even, we can use $d=\theta/2$ edges $\{e_{k+1},\cdots,e_{k+d}\}$ to connect them pairwise, and let each edge corresponds to a matrix $\mb \Sigma^{(k+l)}=\(t_{\alpha,\beta}^{(k+l)}\)$ where $t_{\alpha,\beta}^{(k+l)}=1$ for all $\alpha$ and $\beta$, $l=1,\cdots,d$. Then all vertices of the graph have even degrees.
	Notice that a graph with all its vertices degree being even is a circle, of course two-edges connected. By Lemma \ref{mmlm1}, let $T_1'$ denote the summation that has no restriction on the indicators, use the fact that $\|T^{(k+l)}\|\leq \max_i{m_i}$ and they are rank one matrices, we have
	$$|T'|\leq \(m_i\)^{\max{\{\theta/2,1\}}}\prod_{j=1}^k\|\mb \Sigma^{(j)}\|.$$
	Since the number of odd degree vertices will not become larger by glue any vertices together, this proof of this lemma is complete by induction.
\end{proof}

We now consider the mixed moments
$$M(m_1,m_2,\cdots,m_k):=n^{-\gamma}\E\prod_{s=1}^k(T_s-\E T_s)^{m_s},$$ where $m_1,\cdots,m_k$ are non-negative integers with $\sum_{s=1}^km_s=m.$
We set $\gamma=\sum_{s=1}^ksm_s.$
Now, firstly draw two parallel lines and draw $m_1$ $\mb Q_1$-graphs, denoted as $$\mb{Q}_1(V_1,E_1,F_1), \cdots, \mb{Q}_{1}(V_{m_1},E_{m_1},F_{m_1}).$$
Then draw $m_2$ $\mb Q_2$-graphs, denoted as
$$\mb{Q}_{2}(V_{m_1+1},E_{m_1+1},F_{m_1+1}), \cdots, \mb{Q}_{2}(V_{m_1+m_2},E_{m_1+m_2},F_{m_1+m_2}).$$
Continue this process ending with drawing $m_k$ $\mb Q_k$-graphs, denoted as
$$\mb{Q}_k(V_{\sum_{s=1}^{k-1}m_s+1},E_{\sum_{s=1}^{k-1}m_s+1},F_{\sum_{s=1}^{k-1}m_s+1}), \cdots, \mb{Q}_{k}(V_{m},E_{m},F_{m}).$$  We denote the resulting graph, which is a union of m basic graphs, as $$\mb Q\overset{[m_1,m_2,\cdots,m_k]}{\(\cup_{\ell=1}^{m}V_{\ell},\cup_{\ell=1}^{m}E_{\ell},\cup_{\ell=1}^{m}F_{\ell}\)}.$$ Then we find that
\begin{align}
	&M(m_1,m_2,\cdots,m_k)\\\notag
	=&n^{-\gamma}\sum_{(\cdot)}\sum_{(\cdot\cdot)}\E\prod_{\kappa=1}^{k}\prod_{\theta=1}^{m_k}\Bigg(\mb Q_{\kappa}(V_{\sum_{s=1}^{\kappa-1}sm_{s}+\theta},E_{\sum_{s=1}^{\kappa-1}sm_{s}+\theta},F_{\sum_{s=1}^{\kappa-1}sm_{s}+\theta})\\\notag
	&\quad\quad\quad\quad\quad\quad\quad\quad-\E\mb Q_{\kappa}(V_{\sum_{s=1}^{\kappa-1}sm_{s}+\theta},E_{\sum_{s=1}^{\kappa-1}sm_{s}+\theta},F_{\sum_{s=1}^{\kappa-1}sm_{s}+\theta})\Bigg),
\end{align}
where the summation $\sum_{(\cdot)}$ takes over all possibility of different standard $\mb Q$ graphs and the summation $\sum_{(\cdot\cdot)}$ takes over all possibility of $\mb Q$ graphs that isomorphic with a given standard $\mb Q$ graph.

We classify all the possible resulting standard $\mb Q$ graphs into three categories.
\begin{itemize}
	\item[(1):] The resulting $\mb Q$ graph is called a Type I $\mb Q$-graph if there is at least one single vertical edge in the set $\cup_{\ell=1}^{m}E_{\ell}$.
	\item[(2):] The resulting $\mb Q$ graph is called a Type II $\mb Q$-graph if there is no single vertical edge in the set $\cup_{\ell=1}^{m}E_{\ell}$ but there exist $1\leq l \leq m$ satisfying that $E_{l}\cap\(\cup_{\ell\neq l}E_{\ell}\)=\emptyset$.
	\item[(3):] The resulting $\mb Q$ graph is called a Type III $\mb Q$-graph if it does not belong to the former two categories and there are $l_1,l_2$ and $l_3$ satisfying that $V_{l_1}\cap V_{l_2}\cap V_{l_3}\neq\emptyset$.
	\item[(4):] All other possible resulting $\mb Q$ graphs are classified into Type IV.
\end{itemize}

We find the following easy facts.
\begin{itemize}
	\item[Fact.1:] For all Type I $\mb Q$-graphs, we have $\E(\mb Q-\E\mb Q)$ equal to 0 since the variables in $\mb X_n$ are independent and mean 0.
	\item[Fact.2:] For Type II $\mb Q$-graphs, we have $\E(\mb Q-\E\mb Q)$ equal to 0 since there is a $1\leq l \leq \gamma$ satisfying that $E_{l}\cap\(\cup_{\ell\neq l}E_{\ell}\)=\emptyset$ implies that there exists $1\leq \ell\leq k$ satisfying \begin{align}
		&\E\prod_{\kappa=1}^{k}\prod_{\theta=1}^{m_k}\Bigg(\mb Q_{\kappa}(V_{\sum_{s=1}^{\kappa-1}sm_{s}+\theta},E_{\sum_{s=1}^{\kappa-1}sm_{s}+\theta},F_{\sum_{s=1}^{\kappa-1}sm_{s}+\theta})\\\notag
		&\quad\quad\quad\quad\quad\quad\quad\quad-\E\mb Q_{\kappa}(V_{\sum_{s=1}^{\kappa-1}sm_{s}+\theta},E_{\sum_{s=1}^{\kappa-1}sm_{s}+\theta},F_{\sum_{s=1}^{\kappa-1}sm_{s}+\theta})\Bigg)\\\notag
		=&\E\(\mb Q_{\ell}(V_{l},E_{l},F_{l})-\E\mb Q_{\ell}(V_{l},E_{l},F_{l})\)\\\notag
		&\times\E\Bigg(\prod_{\kappa=1}^{k}\prod_{\overset{\theta=1}{\sum_{s=1}^{\kappa-1}sm_{s}+\theta\neq l}}^{m_k}\bigg(\mb Q_{\kappa}(V_{\sum_{s=1}^{\kappa-1}sm_{s}+\theta},E_{\sum_{s=1}^{\kappa-1}sm_{s}+\theta},F_{\sum_{s=1}^{\kappa-1}sm_{s}+\theta})\\\notag &\quad\quad\quad\quad\quad\quad\quad\quad-\E\mb Q_{\kappa}(V_{\sum_{s=1}^{\kappa-1}sm_{s}+\theta},E_{\sum_{s=1}^{\kappa-1}sm_{s}+\theta},F_{\sum_{s=1}^{\kappa-1}sm_{s}+\theta})\bigg)\Bigg)=0.
	\end{align}
	\item[Fact.3:] For type III $\mb Q$-graphs, by  definition, must have less than $m/2$ connected components (sub-graphs).
	\item[Fact.4:] For type IV $\mb Q$-graphs, by definition, the $m$ basic graphs should be connected with each other by vertical edges pairwisely. That is to say, the number of basic graph in a $\mb Q$-graphs of type IV must be even and this $\mb Q$-graphs must have $m/2$ connected components.
\end{itemize}

We have the following lemmas.
\begin{lemma}\label{T3}
	For any given standard type III $\mb Q$-graph $\mb Q\overset{[m_1,m_2,\cdots,m_k]}{\(\cup_{\ell=1}^{m}V_{\ell},\cup_{\ell=1}^{m}E_{\ell},\cup_{\ell=1}^{m}F_{\ell}\)}$, we have
	\begin{align}
		&\sum_{(\cdot\cdot)}\E\prod_{\kappa=1}^{k}\prod_{\theta=1}^{m_k}\Bigg(\mb Q_{\kappa}(V_{\sum_{s=1}^{\kappa-1}sm_{s}+\theta},E_{\sum_{s=1}^{\kappa-1}sm_{s}+\theta},F_{\sum_{s=1}^{\kappa-1}sm_{s}+\theta})\\\notag
		&\quad\quad\quad\quad\quad\quad\quad\quad-\E\mb Q_{\kappa}(V_{\sum_{s=1}^{\kappa-1}sm_{s}+\theta},E_{\sum_{s=1}^{\kappa-1}sm_{s}+\theta},F_{\sum_{s=1}^{\kappa-1}sm_{s}+\theta})\Bigg)=o(n^{\gamma}).
	\end{align}
\end{lemma}

\begin{proof}
	Notice first that for a Type III $\mb Q$-graph, denoted as $\mb Q$, by definition the degree of any vertex is at least two. For simplification, we denote \begin{align}
		Sum(\mb Q)=&\sum_{(\cdot\cdot)}\E\prod_{\kappa=1}^{k}\prod_{\theta=1}^{m_k}\Bigg(\mb Q_{\kappa}(V_{\sum_{s=1}^{\kappa-1}sm_{s}+\theta},E_{\sum_{s=1}^{\kappa-1}sm_{s}+\theta},F_{\sum_{s=1}^{\kappa-1}sm_{s}+\theta})\\\notag
		&\quad\quad\quad\quad\quad\quad\quad\quad-\E\mb Q_{\kappa}(V_{\sum_{s=1}^{\kappa-1}sm_{s}+\theta},E_{\sum_{s=1}^{\kappa-1}sm_{s}+\theta},F_{\sum_{s=1}^{\kappa-1}sm_{s}+\theta})\Bigg).
	\end{align}
	Denote the pillar of $\mb Q$ as $P(\mb Q)$ and its roof as $R(\mb Q)$. Let $c$ be the number of connected components in $\mb Q$. By Fact.3 we have $c\leq m/2$. Denote the $c$ components as $\mb Q_1,\cdots,\mb Q_c$ and denote their pillars and roofs as $P(\mb Q_1),\cdots, P(\mb Q_c)$ and $R(\mb Q_1),\cdots, R(\mb Q_c)$ respectively. Let $v_{\ell}$ be the number of non-coincident vertices in $P(\mb Q_{\ell})$ for $1\leq \ell\leq c$. Let $e_{\ell}$ be the number of horizontal edges in $P(\mb Q_{\ell})$ for $1\leq \ell\leq c$. Of course, the number of vertical edges should be $2e_{\ell}$ for $1\leq \ell\leq c$. Denote the number of non-coincident multiple edges in  $P(\mb Q_{\ell})$ for $1\leq \ell\leq c$ as $t_{\ell}$.
	
	Since unconnected means there is no common vertex and edge, we shall consider the contributions of each component separately and the whole contribution is the products of all parts.
	Now without losing of generality, we focus on $\mb Q_1$.
	Noting that all the basic graphs are connected graph, and that $\mb Q_1$ consists of at least two basic graphs connected with each other by vertical edges.
	We claim that the contribution of all graphs isomorphic with $\mb Q_1$ is at most $O(\|\mb \Sigma\|^{e_1}n^{e_1})$ if there are two and only two basic graphs in $\mb Q_1$ and $O(\|\mb \Sigma\|^{e_1}n^{e_1-2\varepsilon})$ if there are more than two basic graphs in $\mb Q_1$. In fact, we have the following estimates.
	\begin{itemize}
		\item[(1):]  When there are two and only two basic graphs in $\mb Q_1$. We have the following arguments.
		\begin{itemize}
			\item[(a1):] If $t_1=0$, then all the vertical edges are double edge thus the expectation of the absolute value of the random part can be bounded by $C_{m,k}$, a constant only depend on $k$ and $m$ since the number of terms in the random part depends only on $k$ and $m$ and the expectation of absolute value of any term can be bounded by 1. Also, in a type III $\mb Q$-graph, denoted as $\mb Q$, there must at least a circle in any connected component of its nerve. To see this, we first notice that all basic graph is a circle. And in any connected component of $\mb Q$, there are at least two basic graphs have coincident vertical edges (by definition). The two horizontal edges that connected with these two coincident vertical edges must connect with some different vertical edges and back to the common $J$ vertex. Thus, we have $v_1\leq e_1$ in this situation. Combining with Lemma \ref{roof}, this implies that the contribution of all isomorphic graphs $Sum(\mb Q_1)=O(\|\mb \Sigma\|^{e_1}n^{e_1})$ since every connected component of $R(\mb Q_1)$ is an Euler graph thus must be two edge connected.
			\item[(b1):] If $t_1=1$, then the multiplicity of these coincident edges, denoted as $t_{1}^{(1)}$, must be an even number not smaller than 4. Thus every connected component of $R(\mb Q_1)$ is also an Euler graph thus must be two edge connected. On one hand, the expectation of random parts can be bounded by $C_{m,k}\mu_4(\eta_n\sqrt{n})^{t_1^{(1)}-4}$. On the other hand $v_1$ is at most $e_1-(t_1^{(1)}-4)/2$ since $v_1\leq (2e_1-t_1^{(1)})/2+2$. Thus we have $Sum(\mb Q_1)=O(\|\mb \Sigma\|^{e_1}n^{e_1}).$ We also get from here that the every appearance of vertical edge with multiplicity be even and not smaller than 4 will not increase  the  order of the bound on $Sum(\mb Q_1).$
			\item[(c1):] If $t_1=2$,  denote the multiplicities of these coincident edges as $t_{1}^{(1)}$ and $t_{1}^{(2)}$. When both $t_{1}^{(1)}$ and $t_{1}^{(2)}$ are even, we have $Sum(\mb Q_1)=O(\|\mb \Sigma\|^{e_1}n^{e_1})$ as illustrated in $(b).$
			When both $t_{1}^{(1)}$ and $t_{1}^{(2)}$ are odd, there are two cases. (1): $t_{1}^{(1)}=t_{1}^{(2)}=3$ and (2): one of them equal to 3 and the other one equal or larger than 5. Under case (1), we have the expectation of random parts can be bounded by $C_{m,k}\mu_4^2$ and $v_l\leq e_l$. On the other hand, the $J$ vertices of these coincident edges must coincident and their $I$ vertices must be connected by horizontal edges, for otherwise there must be other multiple edge with odd multiplicity. Notice that the number of vertices that has odd degree in $R(\mb Q)$ is 2, and $v_1\leq e_l-1$ since there is at least a circle in $N(\mb Q_1)$. Then we obtain by Lemma \ref{roof} that $$Sum(\mb Q_1)=p\times O(\|\mb \Sigma\|^{e_1}n^{e_1-1})=O(\|\mb \Sigma\|^{e_1}n^{e_1}).$$ Under case (2), we assume $t_{1}^{(1)}=3$ without loss of generality. We have the expectation of random parts can be bounded by $\mu_4(\eta_n\sqrt{n})^{t_1^{(2)}-4}=O(\sqrt{n})^{t_1^{(2)}-4}$ and $v_l\leq e_1-(t_1^{(2)}-5)/2-1$ by the same argument. Thus, we also have
			$$Sum(\mb Q_1)=p\times o(\|\mb \Sigma\|^{e_1}n^{e_1-1})=O(\|\mb \Sigma\|^{e_1}n^{e_1}).$$ We also get from here that the every appearance of two non-coincident vertical edges with odd multiplicities will not increase  the  order of the bound on $Sum(\mb Q_1).$
		\end{itemize}
		\item[(2):]  When there are more than two basic graphs in $\mb Q_1$. Obviously, there must be three basic graphs connected by their vertical edges (i.e. any two of them have coincident edges), or there are two different pairs of basic graphs connected by their vertical edges. We have the following arguments.
		\begin{itemize}
			\item[(a2):] If $t_1=0$ and there are two different pairs of basic graphs connected by their vertical edges, for the same reason with (a1) we know there are at least two different circles in $N(\mb Q_1)$. If $t_1=0$ and there are three basic graphs connected by their vertical edges, they must be connected by at least two non-coincident vertical edges since there is no multiple edge. Then there are also at least two different circles in $N(\mb Q_1)$. Thus, we have $v_1\leq e_1-1$ and $Sum(\mb Q_1)=O(\|\mb \Sigma\|^{e_1}n^{e_1-1}).$
			\item[(b2):] If $t_1>0$ and there are two different pairs of basic graphs connected by their vertical edges, we have from (a2), (a3) and (b1) that $Sum(\mb Q_1)=O(\|\mb \Sigma\|^{e_1}n^{e_1-1}).$
			\item[(c2):] If $t_1>0$ and there are three basic graphs connected by their vertical edges, and they are connected by at least two non-coincident vertical edges. We have by (a2), (a3) and (b1) that $Sum(\mb Q_1)=O(\|\mb \Sigma\|^{e_1}n^{e_1-1})$ since there are at least two different circles in $N(\mb Q_1)$ and thus $v_1\leq e_1-1$. Otherwise, either there will be a multiple edge and there is at least a circle in $N(\mb Q)$, or there will be a vertical edge with multiplicity not smaller than 6. Thus, we have  $Sum(\mb Q_1)=O(\|\mb \Sigma\|^{e_1}n^{e_1-2\varepsilon}).$
		\end{itemize}
	\end{itemize}
	Combining the argument (a1)-(c1) and (a2) and (c2) and notice that the non-coincident vertical edges with odd multiplicities will always appear pairwise, we complete the proof of the claim.
	
	Thus, we arrive at
	\begin{align}
		Sum(\mb Q)=\prod_{\ell=1}^c Sum (\mb Q_{\ell})=O(\|\mb \Sigma\|^{\gamma}n^{\gamma-2\varepsilon})=O(n^{\gamma-\varepsilon}).
	\end{align}
\end{proof}

\begin{lemma}\label{T4}
	For given $m_1,\cdots,m_k$, we have
	\begin{align}
		&n^{-\gamma}\sum_{(IV)}\sum_{(\cdot\cdot)}\E\prod_{\kappa=1}^{k}\prod_{\theta=1}^{m_k}\Bigg(\mb Q_{\kappa}(V_{\sum_{s=1}^{\kappa-1}sm_{s}+\theta},E_{\sum_{s=1}^{\kappa-1}sm_{s}+\theta},F_{\sum_{s=1}^{\kappa-1}sm_{s}+\theta})\\\notag
		&\quad\quad\quad\quad\quad\quad\quad\quad-\E\mb Q_{\kappa}(V_{\sum_{s=1}^{\kappa-1}sm_{s}+\theta},E_{\sum_{s=1}^{\kappa-1}sm_{s}+\theta},F_{\sum_{s=1}^{\kappa-1}sm_{s}+\theta})\Bigg)\\\notag
		=&{\mathbb MN}[\mb 0,{\mb\Lambda}](m_1,\cdots,m_k)+o(1),
	\end{align}
	where $\sum_{(IV)}$ takes all possible standard Type IV $\mb Q$-graph and ${\mathbb M}(m_1,\cdots,m_k)$ is the mixed moment of a certain multivariate normal distribution with zero mean vector and covariance matrix $\mb \Lambda=\(\sigma_{\phi,\psi}\)_{k\times k}$.
\end{lemma}
\begin{proof}
	By Fact.4, in a Type IV graph, the $\gamma$ basic graphs should be connected by vertical edges pair-wisely. Also, for any $1\leq \phi,\psi\leq k$, we have that $$\sum_{(g(\phi,\psi))}\sum_{(\cdot\cdot)}\E\(\mb{Q}_{\phi}(V_{\theta},E_{\theta},F_{\theta})-\E\mb{Q}_{\phi}(V_{\theta},E_{\theta},F_{\theta})\)\(\mb{Q}_{\psi}(V_{\omega},E_{\omega},F_{\omega})-\E\mb{Q}_{\psi}(V_{\omega},E_{\omega},F_{\omega})\),$$
	where $\sum_{(g(\phi,\psi))}$ takes all possible standard Type IV $\mb Q$-graph for $\mb Q'=\mb Q_{\phi}\cup \mb Q_{\psi}$, $\sum_{s=1}^{\phi-1}sm_{s}\leq \theta\leq \sum_{s=1}^{\phi}sm_{s}$ and $\sum_{s=1}^{\psi-1}sm_{s}\leq \omega\leq \sum_{s=1}^{\psi}sm_{s}$, is the same with $$\sigma_{\phi,\psi}=\E\(L_{\phi}-\E L_{\phi}\)\(L_{\psi}-\E L_{\psi}\).$$
	Suppose in a standard Type IV graph $\mb Q$, the $\gamma$ basic graphs, denoted for simplicity as $\mb Q_1^{1}, \mb Q_1^{2},\cdots, \mb Q_1^{m_1},\cdots, \mb Q_k^{m_k}$ are grouped into $\gamma/2$ pairs $$G:=\{\(\mb Q_{l_1,1}^{\ell_1,1},\mb Q_{l_1,2}^{\ell_1,2}\),\(\mb Q_{l_2,1}^{\ell_2,1},\mb Q_{l_2,2}^{\ell_2,2}\),\cdots, \(\mb Q_{l_{\gamma/2},1}^{\ell_{\gamma/2},1},\mb Q_{l_{\gamma/2},2}^{\ell_{\gamma/2},2}\)\}.$$
	Comparing
	\begin{align}
		&\sum_{(g(G))}\sum_{(\cdot\cdot)}\prod_{s=1}^{\gamma/2}\E\(\mb Q_{l_{s,1}}^{\ell_{s,1}}-\E\mb Q_{l_{s,1}}^{\ell_{s,1}}\)\(\mb Q_{l_{s,2}}^{\ell_{s,2}}-\E\mb Q_{l_{s,2}}^{\ell_{s,2}}\)	
	\end{align}
	where $\sum_{(g(G))}$  takes all possible standard Type IV $\mb Q$-graph for given $G$,
	with
	$$\prod_{s=1}^{\gamma/2}\sigma_{l_{s,1},l_{s,2}}=\prod_{s=1}^{\gamma/2}\sum_{(g({l_{s,1},l_{s,2}}))}\sum_{(\cdot\cdot)}\E\(\mb Q_{l_{s,1}}^{\ell_{s,1}}-\E\mb Q_{l_{s,1}}^{\ell_{s,1}}\)\(\mb Q_{l_{s,2}}^{\ell_{s,2}}-\E\mb Q_{l_{s,2}}^{\ell_{s,2}}\),$$ we find that the latter one has more terms than the former one. However, the corresponding $Q$-graph linked with those additional terms is all Type III. Thus, by Lemma \ref{T3}, we obtain that
	$$\sum_{(g(G))}\sum_{(\cdot\cdot)}\prod_{s=1}^{\gamma/2}\E\(\mb Q_{l_{s,1}}^{\ell_{s,1}}-\E\mb Q_{l_{s,1}}^{\ell_{s,1}}\)\(\mb Q_{l_{s,2}}^{\ell_{s,2}}-\E\mb Q_{l_{s,2}}^{\ell_{s,2}}\)=\prod_{s=1}^{\gamma/2}\sigma_{l_{s,1},l_{s,2}}+o(n^{\gamma}).$$
	Then we get a conclusion that
	\begin{align}
		&n^{-\gamma}\sum_{(IV)}\sum_{(\cdot\cdot)}\E\prod_{\kappa=1}^{k}\prod_{\theta=1}^{m_k}\Bigg(\mb Q_{\kappa}(V_{\sum_{s=1}^{\kappa-1}sm_{s}+\theta},E_{\sum_{s=1}^{\kappa-1}sm_{s}+\theta},F_{\sum_{s=1}^{\kappa-1}sm_{s}+\theta})\\\notag
		&\quad\quad\quad\quad\quad\quad\quad\quad-\E\mb Q_{\kappa}(V_{\sum_{s=1}^{\kappa-1}sm_{s}+\theta},E_{\sum_{s=1}^{\kappa-1}sm_{s}+\theta},F_{\sum_{s=1}^{\kappa-1}sm_{s}+\theta})\Bigg)\\\notag
		=&\sum_{G}\prod_{s=1}^{\gamma/2}\sigma_{l_{s,1},l_{s,2}}+o(1),
	\end{align}
	where $\sum_{G}$ takes all possible of grouping $m_1$ 1's, $m_2$ 1's, $\cdots$, $m_k$ k's pairwise into $$\{(l_{1,1},l_{1,2}),(l_{2,1},l_{2,2}),\cdots,(l_{\gamma/2,1},l_{\gamma/2,2})\}.$$
	By Corollary 4.1 in \cite{jonsson1982some}, we know that it is exact the mixed moment of a certain multivariate normal distribution with zero mean vector and covariance matrix $\mb \Lambda=\(\sigma_{\phi,\psi}\)_{k\times k}$.
	
	This completes the proof of this lemma.
\end{proof}
\subsection{Complete the proof of Theorem \ref{kjoint}}
Adopt the notations in last subsection. When $\gamma$ is odd, then there will be no Type IV graph and thus $n^{-\gamma}\E\prod_{s=1}^k(L_s-\E L_s)^{m_s}=o(1)\sum_{(\cdot)}\leq o(1)C_{m,k}=o(1)$ by Fact.1, Fact.2 and Lemma \ref{T3}. When $\gamma$ is even, then by Lemma \ref{T4} we know that $n^{-\gamma}\E\prod_{s=1}^k(L_s-\E L_s)^{m_s}={\mathbb MN}[\mb 0,{\mb\Lambda}](m_1,\cdots,m_k)+o(1),$
where ${\mathbb M}(m_1,\cdots,m_k)$ is the mixed moment of a certain multivariate normal distribution with zero mean vector and covariance matrix $\mb \Lambda=\(\sigma_{\phi,\psi}\)_{k\times k}$.
Thus, we conclude the mixed moment $n^{-\gamma}\E\prod_{s=1}^k(L_s-\E L_s)^{m_s}$ will always convergence to the mixed moment of a certain multivariate normal distribution with zero mean vector and covariance matrix $\mb \Lambda=\(\sigma_{\phi,\psi}\)_{k\times k}$ and thus we are done.

\section{Some Lemmas}
This section is to give some lemmas that is useful in the proofs.
The following two lemmas follow from straightforward calculation thus we omit their proofs.
\begin{lemma}\label{lm1}
  Let $\mb x=\(x_1,x_2,\cdots,x_n\)'$ be a random vector, where $x_i$s' are i.i.d with mean zero, variance 1, then for any $n\times n$ matrices $\mb A=(a_{ij})$ and $\mb B=(b_{ij})$, we have
  \begin{align*}
    &\E\(\mb z'\mb A\mb z-\tr \mb A\)\(\mb z'\mb B\mb z-\tr \mb B\)
    =\(\E|x_1|^4-3\)\tr\(\mb A\du\mb A\)+\tr\(\mb A\mb B'\)+\tr\(\mb A\mb B\).
  \end{align*}
\end{lemma}

\begin{lemma}\label{lm2}
  Let $\mb x=\(x_1,x_2,\cdots,x_n\)'$ be a random vector, where $x_i$s' are i.i.d with mean zero, variance 1, then for any $n\times n$ matrices $\mb A=(a_{ij})$, we have
  \begin{align*}
    &\E\(\sum_{i=1}^n\(\sum_{j=1}^n a_{i,j}x_j\)^4\)=3\tr\(\mb A'\mb A\du\mb A'\mb A\)+\(\E|x_1|^4-3\)\tr\(\mb A'\du\mb A'\)\(\mb A\du\mb A\).
  \end{align*}
\end{lemma}

\begin{lemma}\label{lm3}
  Let $\mb x=\(x_1,x_2,\cdots,x_n\)'$ be a random vector, where $x_i$ i.i.d with mean zero, variance 1, then for any symmetric $n\times n$ matrices $\mb T=(t_{ij})$ and $\mb W=(w_{ij})$, we have
  \begin{align*}
    &\E\(\mb x'\mb T\mb x\mb x'\mb T\mb x\mb x'\mb W\mb x\)\\\notag
=&\(\tr\mb T\)^2\tr\mb W+2\tr\(\mb T\)^2\tr\mb W+4\tr\(\mb T\)\tr\mb T\mb W+8\tr\(\mb T\mb T\mb W\)\\\notag
&+(\mu_4-3)\(2\tr\(\mb T\du\mb W\)\tr \mb T+\tr\(\mb T\du\mb T\)\tr\mb W\)+(4\mu_4-20)\tr\(\mb T\mb D_{\mb W}\mb T\)\\\notag
&+(8\mu_4-16)\tr\(\mb W\mb D_{\mb T}\mb T\)+\mu_3^2\(4d_{\mb T}'\mb T d_{\mb W}+2d_{\mb T}'\mb W d_{\mb T}+4\mb 1'\(\mb T\du \mb T\du\mb W\)\mb 1\)\\\notag
&+\(\mu_6-15\mu_4-10\mu_3^2+30\)\tr\(\mb T\du \mb T\du\mb W\).
  \end{align*}
\end{lemma}

\begin{proof}
Notice that
\begin{align*}
  \E\(\mb x'\mb T\mb x\mb x'\mb T\mb x\mb x'\mb W\mb x\)=\E\sum_{i_1,\cdots,i_6}x_{i_1}\cdots x_{i_6}t_{i_1i_2}t_{i_3i_4}w_{i_5i_6}.
\end{align*}
We separate the expectation into four parts.
\begin{description}
  \item[part 1] We use $\E(2,2,2)$ to denote the part of the expectation that contains all the terms in which the sequence $\(i_1,\cdots,i_6\)$ has three different entries, each appears twice.
Then we have
\begin{align*}
  \E(2,2,2)=\E(2,2,2)_1+\E(2,2,2)_2+\E(2,2,2)_3+\E(2,2,2)_4,
\end{align*}
where
\begin{align*}
  &\E(2,2,2)_1=\sum_{i_1\neq i_2\neq i_3}t_{i_1i_1}t_{i_2i_2}w_{i_3i_3}\\\notag
&=\sum_{i_1,i_2,i_3}t_{i_1i_1}t_{i_2i_2}w_{i_3i_3}-2\sum_{i_1\neq i_2}t_{i_1i_1}t_{i_2i_2}w_{i_1i_1}-\sum_{i_1\neq i_2}t_{i_1i_1}^2w_{i_2i_2}+2\sum_{i_1}t_{i_1i_1}^2w_{i_1i_1}\\\notag
&=\sum_{i_1,i_2,i_3}t_{i_1i_1}t_{i_2i_2}w_{i_3i_3}-2\sum_{i_1,i_2}t_{i_1i_1}t_{i_2i_2}w_{i_1i_1}-\sum_{i_1,i_2}t_{i_1i_1}^2w_{i_2i_2}
+2\sum_{i_1}t_{i_1i_1}^2w_{i_1i_1}\\\notag
&=\(\tr\mb T\)^2\tr\mb W-2\tr\(\mb T\du\mb W\)\tr \mb T-\tr\(\mb T\du\mb T\)\tr\mb W+2\tr\(\mb T\du \mb T\du\mb W\),
\end{align*}

\begin{align*}
  &\E(2,2,2)_2=2\sum_{i_1\neq i_2\neq i_3}t_{i_1i_2}^2w_{i_3i_3}\\\notag
&=2\sum_{i_1,i_2,i_3}t_{i_1i_2}^2w_{i_3i_3}-4\sum_{i_1,i_2}t_{i_1i_2}^2w_{i_1i_1}-2\sum_{i_1, i_2}t_{i_1i_1}^2w_{i_2i_2}+4\sum_{i_1}t_{i_1i_1}^2w_{i_1i_1}\\\notag
&=2\tr\(\mb T\)^2\tr\mb W-4\tr\(\mb T\mb D_{\mb W}\mb T\)-2\tr\(\mb T\du\mb T\)\tr\mb W+4\tr\(\mb T\du \mb T\du\mb W\),
\end{align*}

\begin{align*}
  &\E(2,2,2)_3=4\sum_{i_1\neq i_2\neq i_3}t_{i_1i_1}t_{i_2,i_3}w_{i_3i_3}\\\notag
&=4\sum_{i_1,i_2,i_3}t_{i_1i_2}^2w_{i_3i_3}-8\sum_{i_1,i_2}t_{i_1i_1}t_{i_1i_2}w_{i_1i_2}\\
&\quad\quad-4\sum_{i_1, i_2}t_{i_1i_1}t_{i_2i_2}w_{i_2i_2}+8\sum_{i_1}t_{i_1i_1}^2w_{i_1i_1}\\\notag
&=4\tr\(\mb T\)\tr\mb T\mb W-8\tr\(\mb W\mb D_{\mb T}\mb T\)-4\tr\mb T\tr\(\mb T\du\mb W\)+8\tr\(\mb T\du \mb T\du\mb W\),
\end{align*}
and
\begin{align*}
  &\E(2,2,2)_4=8\sum_{i_1\neq i_2\neq i_3}t_{i_1i_2}t_{i_2,i_3}w_{i_3i_1}\\\notag
&=8\sum_{i_1,i_2,i_3}t_{i_1i_2}t_{i_2,i_3}w_{i_3i_1}-16\sum_{i_1,i_2}t_{i_1i_2}w_{i_2,i_2}t_{i_2i_1}\\
&\quad\quad-8\sum_{i_1, i_2}t_{i_1i_1}t_{i_2i_1}w_{i_1i_2}+16\sum_{i_1}t_{i_1i_1}^2w_{i_1i_1}\\\notag
&=8\tr\(\mb T\mb T\mb W\)-16\tr\(\mb T\mb D_{\mb W}\mb T\)-8\tr\(\mb W\mb D_{\mb T}\mb T\)+16\tr\(\mb T\du \mb T\du\mb W\).
\end{align*}

  \item[part 2] We use $\E(2,4)$ to denote the part of the expectation that contains all the terms in which the sequence $\(i_1,\cdots,i_6\)$ has two different entries, one appears twice, the other appears four times.
The same, we have
\begin{align*}
  &\E(2,4)\\
  =&\mu_4\Bigg(2\sum_{i_1\neq i_2}t_{i_1i_1}t_{i_2,i_2}w_{i_2i_2}+\sum_{i_1\neq i_2}t_{i_1i_1}^2w_{i_2i_2}\\
  &\quad+8\sum_{i_1\neq i_2}t_{i_1i_1}t_{i_1,i_2}w_{i_1i_2}+4\sum_{i_1\neq i_2}t_{i_1i_2}^2w_{i_2i_2}\Bigg)\\\notag
=&\mu_4\Bigg(2\tr\mb T\(\mb T\du\mb W\)+\tr\(\mb T\du\mb T\)\tr \mb W+8\tr\(\mb W\mb D_{\mb T}\mb T\)\\
&\quad+4\tr\(\mb T\mb D_{\mb W}\mb T\)-15\tr\(\mb T\du \mb T\du\mb W\)\Bigg).
\end{align*}

  \item[part 3] We use $\E(3,3)$ to denote the part of the expectation that contains all the terms in which the sequence $\(i_1,\cdots,i_6\)$ has two different entries, both appear three times.
We have
\begin{align*}
  &\E(3,3)=\mu_3^2\(4\sum_{i_1\neq i_2}t_{i_1i_1}t_{i_1,i_2}w_{i_2i_2}+2\sum_{i_1\neq i_2}t_{i_1i_1}t_{i_2i_2}w_{i_1i_2}+4\sum_{i_1\neq i_2}t_{i_1,i_2}^2w_{i_1i_2}\)\\\notag
&=\mu_3^2\(4d_{\mb T}'\mb T d_{\mb W}+2d_{\mb T}'\mb W d_{\mb T}+4\mb 1'\(\mb T\du \mb T\du\mb W\)\mb 1-10\tr\(\mb T\du \mb T\du\mb W\)\).
\end{align*}

  \item[part 4] We use $\E(6)$ to denote the part of the expectation that contains all the terms in which all the entries in the sequence $\(i_1,\cdots,i_6\)$ are same.
We have
\begin{align*}
  &\E(6)=\mu_6\tr\(\mb T\du \mb T\du\mb W\).
\end{align*}

Combining the argument above, we complete the proof of this lemma.
\end{description}
\end{proof}
The next lemma is the well known central limit theorem for martingale.
\begin{lemma}[Theorem 35.12 of \cite{Billingsley95P}]\label{cltmar}
Suppose that for each $n$, $Y_{n 1}, Y_{n 2}, \cdots, Y_{n r_{n}}$ is a real martingale difference sequence with respect to the increasing $\sigma$-field $\left\{\mathcal{F}_{n j}\right\}$ having second moments. If, as $n\to \infty$,
$$\sum_{j=1}^{r_{n}} \mathrm{E}\left(Y_{n j}^{2} | \mathcal{F}_{n, j-1}\right) \stackrel{i . p .}{\longrightarrow} \sigma^{2},$$
where $\sigma^2$  is a positive constant, and, for each $\varepsilon>0,$
$$
\sum_{j=1}^{r_{n}} \mathrm{E}\left(Y_{n j}^{2} I_{\left(\left|Y_{n j}\right| \geq \varepsilon\right)}\right) \rightarrow 0,
$$
then
$$\sum_{j=1}^{r_{n}} Y_{n r_{n}} \stackrel{\mathcal{D}}{\rightarrow} N\left(0, \sigma^{2}\right).$$
\end{lemma}

\vskip 14pt
\noindent {\large\bf Acknowledgments}

\noindent {The author would like to thank Prof. Weiming Li for his constructive suggestions on the organization of this paper.}

%\bibliographystyle{apalike}
%\bibliography{reference,ref,supp,paper,book}

\begin{thebibliography}{99}
	
	\bibitem[Bai and Silverstein, 2010]{BSbook}
	Bai, Z.~D. and Silverstein, J.~W. (2010).
	\newblock {\em Spectral analysis of large dimensional random matrices}.
	\newblock Springer, 2nd edition.
	
	\bibitem[Bai and Silverstein, 2004]{BS04}
	Bai, Z.~D. and Silverstein, J.~W. (2004).
	\newblock Central limit theorem for linear spectral statistics of large
	dimensional sample covariance matrices.
	\newblock {\em The Annals of Probability}, 32(1A):553--605.
	
	\bibitem[Baltagi et~al., 2017]{baltagi2017asymptotic}
	Baltagi, B.~H., Kao, C., and Wang, F. (2017).
	\newblock Asymptotic power of the sphericity test under weak and strong factors
	in a fixed effects panel data model.
	\newblock {\em Econometric Reviews}, 36(6-9):853--882.
	
	\bibitem[Billingsley, 1995]{Billingsley95P}
	Billingsley, P. (1995).
	\newblock {\em {Probability and measure}}.
	\newblock John Wiley\&Sons, New York.
	
	\bibitem[Gao et~al., 2017]{gao2017high}
	Gao, J., Han, X., Pan, G., and Yang, Y. (2017).
	\newblock High dimensional correlation matrices: the central limit theorem and
	its applications.
	\newblock {\em Journal of the Royal Statistical Society: Series B (Statistical
		Methodology)}, 79(3):677--693.
	
	\bibitem[Hu et~al., 2019]{huli19}
	Hu, J., Li, W., Liu, Z., and Zhou, W. (2019).
	\newblock High-dimensional covariance matrices in elliptical distributions with
	application to spherical test.
	\newblock {\em Ann. Statist.}, 47(1):527--555.
	
	\bibitem[Johnstone, 2001]{Johnstone01}
	Johnstone, I.~M. (2001).
	\newblock On the distribution of the largest eigenvalue in principal components
	analysis.
	\newblock {\em The Annals of Statistics}, 29(2):295--327.
	
	\bibitem[Jonsson, 1982]{jonsson1982some}
	Jonsson, D. (1982).
	\newblock Some limit theorems for the eigenvalues of a sample covariance
	matrix.
	\newblock {\em Journal of Multivariate Analysis}, 12(1):1--38.
	
	\bibitem[Ledoit et~al., 2002]{ledoit2002some}
	Ledoit, O., Wolf, M., et~al. (2002).
	\newblock Some hypothesis tests for the covariance matrix when the dimension is
	large compared to the sample size.
	\newblock {\em The Annals of Statistics}, 30(4):1081--1102.
	
	\bibitem[Pan and Zhou, 2008]{PZ08}
	Pan, G.~M. and Zhou, W. (2008).
	\newblock Central limit theorem for signal-to-interference ratio of reduced
	rank linear receiver.
	\newblock {\em The Annals of Applied Probability}, 18(3):1232--1270.
	
	\bibitem[Schott, 2005]{Schott05}
	Schott, J.~R. (2005).
	\newblock Testing for complete independence in high dimensions.
	\newblock {\em Biometrika}, 92(4):951--956.
	
	\bibitem[Srivastava, 2005]{srivastava2005some}
	Srivastava, M.~S. (2005).
	\newblock Some tests concerning the covariance matrix in high dimensional data.
	\newblock {\em Journal of the Japan Statistical Society}, 35(2):251--272.
	
	\bibitem[Yang and Pan, 2015]{yang15}
	Yang, Y. and Pan, G. (2015).
	\newblock Independence test for high dimensional data based on regularized
	canonical correlation coefficients.
	\newblock {\em The Annals of Statistics}, 43(2):467--500.
	
	\bibitem[Zheng et~al., 2015]{zheng2015substitution}
	Zheng, S., Bai, Z., and Yao, J. (2015).
	\newblock Substitution principle for clt of linear spectral statistics of
	high-dimensional sample covariance matrices with applications to hypothesis
	testing.
	\newblock {\em The Annals of Statistics}, 43(2):546--591.
	
\end{thebibliography}

\end{document}